\documentclass[11pt, a4paper]{article}
\usepackage{amsfonts}
\usepackage{amsmath}
\usepackage{amsthm}
\usepackage{amscd}
\usepackage{amsmath}
\usepackage{amssymb}
\usepackage{amscd}
\usepackage{amsfonts}
\usepackage{amsbsy}
\usepackage{graphicx}
\usepackage{booktabs}
\usepackage{makecell}
\usepackage{enumitem}
\usepackage{tikz}
\usepackage{booktabs}
\usepackage{multirow}
\usepackage{diagbox}
\usepackage{appendix}
\usepackage{mathrsfs}
\usepackage{amsmath,amsfonts,amsthm,amssymb,amscd,color,xcolor}
\usepackage{graphicx,amscd,mathrsfs,wrapfig,mathrsfs,lipsum}
\usepackage{microtype}
\usepackage{float}
\usepackage{tikz}
\usepackage{multicol}
\usepackage{caption}
\usetikzlibrary{arrows}
\usepackage{hyperref}
\usepackage{comment}

\usepackage{hyperref}

\hypersetup{colorlinks=true,citecolor=blue,%
        linkcolor=blue,urlcolor=blue,bookmarksnumbered=true,
        bookmarksopen=true,bookmarksopenlevel=1,breaklinks=true}

\newcommand{\R}{\ensuremath{\mathbb{R}}}

\newcommand{\N}{\ensuremath{\mathbb{N}}}

\newcommand{\Z}{\ensuremath{\mathbb{Z}}}

\newcommand{\la}{\lambda}

\newcommand{\e}{\varepsilon}
\def\epsilon{\varepsilon}
\newcommand{\supp}{\mathop{\rm supp}\nolimits}
\newcommand{\leb}{\mathop{\rm Leb}\nolimits}
\definecolor{vngreen}{rgb}{0.0,0.45,0.0}

\newcommand{\dd}{{\textup d}}
\newcommand{\Id}{{\textup{Id}}}

\newcommand{\pP}{{\mathbb P}}

\newcommand{\PPPP}{{\mathfrak P}}
\newcommand{\Gl}{{\textup{GL}}}
\newcommand{\Sl}{{\textup{SL}}}
\newcommand{\Tv}{{\textup{TV}}}
\newcommand{\M}{{\textup{M}}}
\newcommand{\bOmega}{{\bf {\Omega}}}
\newcommand{\bomega}{{\boldsymbol{\omega}}}
\newcommand{\bFF}{{\boldsymbol{\mathcal{F}}}}
\newcommand{\bpP}{{{\textup{\bf P}}}}
\newcommand{\dett}{{\textup{det}}}
\newcommand{\TTT}{{\textup{T}}}
\newcommand{\XX}{\mathcal{X}}
\newcommand{\OO}{{\cal O}}
\newcommand{\NN}{{\cal N}}
\newcommand{\LLLL}{{\cal L}}

\newcommand{\Leb}{\mathop{\mathrm{Leb}}\nolimits}
\newcommand{\ZZ}{{\cal Z}}
\newcommand{\dist}{\mathop{\rm dist}\nolimits}

\let\emptyset\varnothing

\newtheorem*{mt}{Main Theorem}
\newtheorem{theorem}{Theorem}[section]
\newtheorem{proposition}[theorem]{Proposition}
\newtheorem{lemma}[theorem]{Lemma}

\newtheorem{corollary}[theorem]{Corollary}
\newtheorem{remark}[theorem]{Remark}
\theoremstyle{definition}
\newtheorem{definition}[theorem]{Definition}

\numberwithin{equation}{section}
\renewcommand{\div}{{\mathrm{div}}}

\numberwithin{figure}{section}
\def\9{{\infty}}

\def\bbp{{\mathbb{P}}}
\def\bbr{{\mathbb{R}}}
\def\bbt{{\mathbb{T}}}

\def\na{{\nabla}}
\def\ve{{\varepsilon}}

\def\hH{{\mathscr{H}}}
\def\mM{{\mathscr{M}}}
\def\vV{{\mathscr{V}}}
\def\eE{{\mathscr{E}}}

\newcommand{\lag}{\langle}
\newcommand{\rag}{\rangle}

\def\wt{\widetilde}
\def\ol{\overline}
\def\wh{\widehat}
\def\ol{\overline}
\def\({\left(}
\def\){\right)}
\def\<{\langle}
\def\>{\rangle}

\newcommand{\ty}{\infty}
\newcommand{\LL}{{\mathcal{L}}}
\newcommand{\E}{{\mathbb{E}}}
\newcommand{\aA}{{\mathcal{A}}}

\def\DD{{\mathcal{D}}}
\def\HH{{\mathcal{H}}}
\def\TT{{\mathcal{T}}}
\def\PP{{\mathcal{P}}}
\def\KK{{\mathcal{K}}}
\def\VV{{\mathcal{V}}}

\newcommand{\FFFF}{{\mathfrak F}}

\colorlet{darkblue}{blue!50!black}

\hypersetup{
    colorlinks,%
    citecolor=blue,%
    filecolor=red,%
    linkcolor=darkblue,%
    urlcolor=blue,%
    pdfnewwindow=true,%
    pdfstartview={FitH}
}

\begin{document}

\author{Vahagn~Nersesyan\,\footnote{NYU-ECNU Institute of Mathematical Sciences at NYU Shanghai, 3663 Zhongshan Road North, Shanghai, 200062, China, e-mail: \href{mailto:vahagn.nersesyan@nyu.edu}{Vahagn.Nersesyan@nyu.edu}}         \and 
  Deng Zhang\,\footnote{School of Mathematical Sciences, Shanghai Jiao Tong University, 200240, Shanghai, China, e-mail: \href{mailto:dzhang@sjtu.edu.cn} 
  {dzhang@sjtu.edu.cn}. 
  Corresponding author}
       \and Chenwan Zhou\,\footnote{NYU-ECNU Institute of Mathematical Sciences at NYU Shanghai, 3663 Zhongshan Road North, Shanghai, 200062, China, e-mail: \href{mailto:cz3557@nyu.edu}{cz3557@nyu.edu}}
}

\title{On the chaotic behavior of the Lagrangian flow of the 2D Navier--Stokes system with bounded degenerate noise}
\date{\today}
\maketitle

\begin{abstract}

We consider a fluid governed by the randomly forced 2D Navier--Stokes system. The force is bounded and highly degenerate, acting directly only on a small number of Fourier modes, and satisfies natural decomposability and observability properties. We prove that the Lagrangian flow associated with this system exhibits chaotic behavior, characterized by the strict positivity of the top Lyapunov exponent. To~achieve this, we develop a new abstract criterion for random dynamical systems that enables one to derive the positivity of the top Lyapunov exponent and exponential mixing from controllability properties of the associated nonlinear and linearized deterministic systems. In particular, our mixing criterion removes the analyticity assumption required in previous works and establishes an exponential rate of convergence in the case where the unperturbed dynamics does not admit a unique globally stable equilibrium.



\end{abstract}

 \tableofcontents

\section{Introduction}\label{Sec-Itro-Main}

 The study of chaotic behavior of dynamical systems has been an important topic in recent decades. It is widely accepted that chaos is the sensitive dependence on initial conditions, characterized by the positivity of the top Lyapunov exponent. Given the significant challenges in proving such properties in purely deterministic scenarios, attention has increasingly shifted towards randomly forced systems.  In these systems, the presence of noise has averaging effects, making the analysis of chaotic properties more tractable~\cite{You13, BBPS323}.

\subsection{Lagrangian chaos}

  Chaos plays an important role in the study of fluid dynamics, providing critical insights into the phenomenon of turbulence~\cite{BJPV98}. 
 Despite its significance, there are only a few mathematically rigorous results on fluid~models.

  In this paper, we consider the Lagrangian flow associated with the randomly forced 2D Navier--Stokes system on the torus $\bbt^2:=\R^2/2\pi\Z^2$. The~Lagrangian flow is a family of diffeomorphisms 
  $$
  \phi^t: \bbt^2\rightarrow \bbt^2, \quad x\mapsto y(t), \quad  t\ge 0
  $$ 
  defined through the ODE 
\begin{equation}\label{E:1.1}
  \dot y(t)=u(t,y(t)),\quad y(0)=x\in \bbt^2,
\end{equation}where $u: \R_+\times\bbt^2\rightarrow \R^2$ is a sufficiently smooth velocity field of an incompressible fluid governed by the Navier--Stokes system
\begin{gather}
 \partial_t u - \nu \Delta u + \<u, \na\> u + \na p= \eta,\ \ \div\,  u=0, \quad x\in \bbt^2,\label{E:1.2}\\
	u(0)=u_0. \label{E:1.3}
\end{gather}
 Here $p=p(t,x)$ represents the pressure of the fluid, $\nu>0$ is the kinematic viscosity,  and $\eta$ is a random force (noise) defined on some probability space~$(\Omega_0, \mathcal{F}_0, \bbp_0)$. The~pair~$(u_t, \phi^t)$ contains information about the velocities and positions of fluid particles at time $t$.

Chaoticity of the Lagrangian flow or {\it Lagrangian chaos} in the context of the system~\eqref{E:1.1}, \eqref{E:1.2} refers to the strict positivity of the top Lyapunov exponent for the Lagrangian flow~$\phi^t$. That is, to establish Lagrangian chaos, one seeks to 
 prove the existence of a deterministic constant~$\lambda_+ > 0$  such that the following limit holds:
$$
\lambda_+=\lim_{t\rightarrow \infty}\frac 1t\log|D_x\phi^t| \quad \text{$\bbp_0$-almost surely}
$$ for typical (in some sense) initial data $(u_0, x)$. Here, $D_x\phi^t$ is the Jacobian matrix of $\phi^t$ at the point $x\in \bbt^2$ and $|D_x\phi^t|$ denotes its norm.~Thus, the exponent $\lambda_+$ provides information about the exponential growth of the Jacobian. Heuristically, this means that  
$$
d_{\bbt^2}(\phi^t(x_1),\phi^t(x_2)) \approx e^{\la_+ t} d_{\bbt^2}(x_1,x_2) 
$$for sufficiently close points $x_1,x_2\in \bbt^2$. Therefore, the positivity of $\lambda_+$ indicates exponential sensitivity of Lagrangian trajectories with respect to the initial data, which is the essence of Lagrangian chaos.

The study of Lagrangian chaos for the Navier--Stokes system was initiated in~\cite{BBS22} for the case of a {\it non-degenerate} white-in-time noise. In~that work, the non-degeneracy assumption plays a key role, as it ensures the {\it strong~Feller property}, which is essential for the employed methods. The {\it highly degenerate} (or {\it hypoelliptic}) white-in-time case was recently addressed in~\cite{CR24} using a new method based on smoothing estimates for the transition function of the triple matrix process. This method establishes positivity of the top Lyapunov exponent for the Lagrangian flow without requiring a strong Feller property, thus overcoming a key limitation of earlier methods.

  The main objective of this paper is to demonstrate Lagrangian chaos in the case when the noise~$\eta$ is a ({\it non-Gaussian}) {\it bounded, highly degenerate} process perturbing directly only a small number of low Fourier~modes. In~order to establish this result, we introduce a new {\it abstract criterion} for chaoticity of random dynamical systems, by developing the framework and mixing results of the papers~\cite{KNS20,KNS20.0}.~This criterion provides a new relationship between the positivity of the top Lyapunov exponent and the {\it controllability} properties of both nonlinear and linearized systems. We show that the conditions of this criterion are satisfied for the system \eqref{E:1.1}, \eqref{E:1.2}. Since the noise is degenerate and has bounded support, the verification of the controllability properties is highly non-trivial. The proposed approach is sufficiently general and can be applied
in other contexts. This paper addresses a question raised in~\cite{BBS22} (see Remark~1.11) and~\cite{BBPS323} (see Section~6). 

As a second main result, we generalize the abstract mixing criteria in~\cite{KNS20,KNS20.0}, establishing exponential mixing without assuming analyticity or the existence of a unique globally stable equilibrium. This relaxation is essential for application to the Lagrangian process, which lacks both properties.

\subsection{Main result in the case of a Haar noise}

The main results of this paper apply to the system~\eqref{E:1.1},~\eqref{E:1.2} driven by a broad class of bounded noises. However, in this Introduction, we focus on a particular case where $\eta$ is a {\it Haar coloured noise}. More precisely, we assume the following condition.
 \begin{description}
 \item[Structure of the noise.]
{\sl The process $\eta$ in \eqref{E:1.2} is of the form
$$
	\eta(t,x)=  \sum_{j\in\mathcal{I}} \eta^j(t)E_j(x),
$$ where 
\begin{itemize}
	\item  $\{E_j\}$ are the trigonometric functions  
$$E_j(x)=j^{\bot}
\begin{cases}
\cos\langle j, x  \rangle \ \textrm{for} \ j_1>0 \ \textrm{or} \ j_1=0, \ j_2>0, \\
\sin\langle j, x \rangle \ \textrm{for} \ j_1<0 \ \textrm{or} \ j_1=0, \ j_2<0
\end{cases}
$$with $j^{\bot}=(-j_2, j_1)$,  and  $\mathcal{I}\subset \mathbb{Z}^2_*$ is a finite set given by
$$
	\mathcal{I}=\left\{j=(j_1, j_2)\in \mathbb{Z}^2_*: |j_i|\leq 1,\, i=1, 2\right\},
$$
	\item  $\{\eta^j\}$ are   independent copies of a random process  $\tilde\eta$ given~by
$$
	\tilde\eta(t)=\sum_{i=0}^\infty \xi_i h_0(t-i)
	+\sum_{i=1}^\infty c_i\sum_{l=0}^\infty \xi_{il}h_{il}(t),
$$with $\{h_0,h_{il}\}$ being the $L^\ty$-normalized
 Haar system, $c_i=A^{-i}$ with some $A>1$ or
  	$c_i=C i^{-q}$ with some $C>0$ and  $q>1$, 
 and~$\{\xi_i,\xi_{il}\}$ are independent identically distributed (i.i.d.) scalar random variables with a common Lipschitz-continuous density~$\rho$ such that 
\begin{equation}\label{E:1.4}
	  \supp\rho\subset[-1,1] \quad\text{and}\quad \rho(0)>0.
\end{equation} 
\end{itemize}

}
\end{description}

 We consider the Navier--Stokes system \eqref{E:1.2} in the usual space of divergence-free vector fields with zero mean value
 $$
H:=\left\{u\in L^2(\bbt^2, \R^2): \div\,u=0,   \ \int_{\bbt^2}u(x)\dd x=0 \right\}, 
$$as well as in more regular spaces $V^k:=H^k\cap H$, where $H^k:=H^k(\bbt^2, \R ^2)$ is~the Sobolev space of order $k\ge1$ endowed with the norm $\|\cdot\|_{H^k}$. For~any~$u_0\in H$, the problem~\eqref{E:1.2}, \eqref{E:1.3} has a unique solution whose restriction to integer times~$u_n$ forms a Markov process. It is proved in~\cite{KNS20} that under the above conditions this process has a unique stationary measure~$\mu$ which is exponentially mixing. The inclusion in \eqref{E:1.4} implies that the
process~$\eta$ is bounded. Combined with the dissipativity and the
parabolic regularization of the Navier--Stokes dynamics, this
ensures that the measure~$\mu$ is concentrated on a set
$X$ that is compact in $V^5$, defined as the closure in a higher-order Sobolev
space of the states reachable from zero; the compactness of~$X$ is
formally established in Section~\ref{S:4.1}, and the equality
$X=\supp\mu$ follows from Corollary~\ref{C:5.6}.

Let $y_n:=y(n)$, $n\ge0$, denote the restriction of the Lagrangian
trajectory \eqref{E:1.1} to integer times. The pair~$(u_n, y_n)$, called the {\it Lagrangian~process}, is well-defined due to the
regularity of the velocity field and is Markovian in the compact metric
space~$X\times \bbt^2$. The incompressibility of the fluid implies that $\mu\times \leb$ is a stationary measure for this process, where $\leb$ denotes the normalized Lebesgue measure on $\bbt^2$. However, the ergodicity of $\mu$ does not automatically guarantee the ergodicity of the product measure $\mu\times \leb$. Establishing the ergodicity of the latter is a challenging task, particularly in the presence of a highly degenerate noise. 

The~following is our main result in the setting of a Haar noise.
 \begin{mt}
Under the above conditions, the measure $\mu\times \leb $ is an exponentially mixing stationary measure for the   Lagrangian process $(u_n, y_n)$. Moreover, there exists a deterministic constant $\lambda_+>0$ such that the following limit~holds:
\begin{equation}\label{E:1.5}
	\lambda_+=\lim_{n\rightarrow \infty}\frac 1n\log|D_x\phi^n| 
\end{equation}
for $\mu\times \leb\times \pP_0$-a.e. $(u_0,x,\omega)\in X\times \bbt^2\times \Omega_0$. 
\end{mt}
  A more general version of this result is presented in Theorem~\ref{T:4.2}, which applies to a broader class of random processes $\eta$ satisfying some decomposability and observability assumptions.

\subsection{Related literature on Lagrangian chaos}

In the physics literature, there has been extensive research on Lagrangian and Eulerian chaos and their interconnections; e.g., see~\cite{AGM-96,BJPV98,CFVP-91,FD-01, GV-94}.   However, the number of mathematically rigorous results remains limited despite the importance of the topic. 

Chaoticity of the Galerkin approximations of the 2D Navier--Stokes system with a degenerate noise has been proved in~\cite{BPS24}. The proof relies on a version of a sufficient condition for chaoticity obtained in~\cite{BBPS-422}, whose hypotheses are verified by using a reduction to genericity properties of a diagonal sub-algebra and some computational algebraic geometry. This approach not only makes it possible to establish the positivity of the top Lyapunov exponent but also furnishes a quantitative lower bound. For the linearized Navier--Stokes system and the advection-diffusion equation, we refer to the recent work \cite{HPRY24}, which shows that the top Lyapunov exponent is bounded from below by a negative power of the diffusion parameter. 
 In the case where the top Lyapunov exponent of the linearization at the zero solution is positive, a quantitative instability has been proved in~\cite{BR24} for stochastic reaction–diffusion equations.

Chaoticity of the Eulerian component (i.e., the velocity field) of the stochastic Navier--Stokes system remains an important open problem. Lagrangian chaos has been established in the paper~\cite{BBS22} for the system \eqref{E:1.1},~\eqref{E:1.2} driven by a white-in-time noise that is non-degenerate at high Fourier modes. More~precisely, the noise is assumed to be of the form
$$
\eta(t,x)=\partial_t W(t,x), \quad  W(t,x)=\sum_{j\in \Z^2_*} b_j \beta_j(t) E_j(x),
$$where $\{\beta_j\}$ is a sequence of independent standard Brownian motions and~$\{b_j\}$ is a sequence of real numbers such that 
\begin{equation}
	\frac{c}{|j|^{\alpha}}  \le |b_j|\le \frac{C}{|j|^{\alpha}}  \quad \text{ for }|j|\ge  L  \label{E:1.6}
\end{equation}
with some constants $c,C, L>0$ and $\alpha>5$.~The non-degeneracy assumption~\eqref{E:1.6} ensures the {\it strong Feller property} for the Lagrangian process~$(u_t, y_t)$, which is essential for the arguments of~\cite{BBS22}. Indeed, the strong Feller property allows one to prove the ergodicity of the Lagrangian process, which, combined with the multiplicative ergodic theorem, guarantees the existence of the exponent~$\la_+$. Since the Lagrangian flow $\phi^t:\bbt^2\to \bbt^2$ is volume-preserving, we have $\la_+\ge0$. The strong Feller property is also used to prove the positivity of~$\la_+$. Indeed, according to Furstenberg's criterion~\cite{Fur63, Led86}, if~$\la_+=0$, then there is a measurable deterministic structure that is almost surely invariant under the dynamics of the triple~$(u_t, y_t, D_x\phi^t)$ (the precise formulation is recalled below in Theorem~\ref{T:2.3}). When the velocity~$u$ is provided by a finite-dimensional system, e.g., by Galerkin approximations of the 2D Navier--Stokes system, the existence of an invariant structure can be ruled out (thus establishing the positivity of~$\la_+$) by showing that the probability law of~$D_x\phi^t$ is sufficiently non-degenerate; for details see~Section~2.4 in~\cite{BBS22}. In the case when~$u$ is infinite-dimensional and is given by the 2D Navier--Stokes system, this strategy does not work anymore.  
Under the strong Feller assumption, the paper~\cite{BBS22} furnishes a refined version of Furstenberg's criterion, with a dichotomy between two possible continuous invariant structures.   Then the existence of such structures is ruled out by using an approximate controllability~property. 

Building on the Lagrangian chaos established in~\cite{BBS22}, the same authors have derived further important results in a series of subsequent papers. These include almost-sure exponential mixing of passive scalars~\cite{BBPS22c} and enhanced dissipation~\cite{BBPS21} for the Navier--Stokes equations driven by a non-degenerate noise. Furthermore, a rigorous proof of a version of Batchelor's law has been given in~\cite{BBPS22d}. In all these papers the noise is white-in-time and non-degenerate with coefficients satisfying \eqref{E:1.6}. 

The recent paper~\cite{CR24} introduces an elegant new method for the construction of a continuous invariant structure that does not require the strong Feller property. Instead, the construction relies on a smoothing estimate for the transition function of the triple matrix process. This estimate is verified for the Navier--Stokes system driven by a highly degenerate white-in-time noise using the Malliavin calculus approach from~\cite{HM06}. Thus, the paper~\cite{CR24} establishes the positivity of the top Lyapunov exponent and exponential mixing of passive scalars.

The present paper focuses on the case of (non-Gaussian) bounded, highly degenerate noises, for which Malliavin calculus techniques are not applicable.

\subsection{Ideas of the proof}

The proof of the Main Theorem can be summarized as follows.

\medskip
   \noindent 
{\bf Exponential mixing of the Lagrangian process.} The first step in the proof of the Main Theorem is the study of the ergodicity of the Lagrangian process $(u_n, y_n)$. Given that the process is considered in a compact phase space, the existence of a stationary measure follows from the classical Bogolyubov--Krylov argument (see, for instance,~\cite{KS12}). The~uniqueness of the stationary measure and the mixing are much more delicate and are derived by developing the controllability approach of the paper~\cite{KNS20} (see~also~\cite{KNS20.0}); in the case of a non-degenerate bounded noise, the~mixing for the process~$(u_n, y_n)$ has been shown~in~\cite{JNPS21}. The abstract result of~\cite{KNS20} requires, in particular, that the dynamics depend analytically on the force and that there is a uniform contraction toward a reference state; in the PDE applications considered therein, the latter corresponds to the existence of a unique globally stable equilibrium, and the resulting mixing rate is exponential. In the subsequent work~\cite{KNS20.0}, the contraction condition is relaxed to global approximate controllability to a point, which, in the applications, allows for finitely many equilibria; however, the mixing rate obtained there is not quantitative. Neither result applies directly to the Lagrangian process: the ODE component~\eqref{E:1.1} does not necessarily depend analytically on the force~$\eta$ in~\eqref{E:1.2}, and every element of the set $\{0\}\times \bbt^2$ is an equilibrium of the unforced system.

We show that both requirements can be dropped, while the mixing rate remains exponential. This answers a question left open in~\cite{KNS20.0}. First, the analyticity assumption is removed completely; in the proofs, it is replaced by an approximation argument that preserves the squeezing properties needed for the coupling. Second, the exponential rate is recovered as follows. In~\cite{KNS20}, the region $\|u-u'\|> d$ is treated using the uniform contraction toward the globally stable equilibrium. With a probability bounded below, the same noise drives both trajectories toward that equilibrium and eventually into a neighborhood of the diagonal. This mechanism is unavailable when the unforced system has several equilibria, since no single equilibrium attracts all trajectories uniformly. In~\cite{KNS20.0}, the contraction is replaced by global approximate controllability to a point, which ensures the recurrence of two independent trajectories to a common small ball. Once the trajectories are sufficiently close, an approximate coupling based on the controllability of the linearized dynamics is combined with a conditional random-walk argument; this gives asymptotic stability, but without a quantitative rate. In the present paper, we introduce a finite partition $\{\mathbb{X}_k\}_{k=-N}^{-1}$ of the large-distance region, defined in terms of successive controllability steps toward a common target point. With a uniform positive probability, a pair of trajectories moves to a higher level of the partition and eventually enters a neighborhood of the diagonal, where the local squeezing coupling applies. Together with a suitable Kantorovich functional, these properties yield a strict contraction of the Markov operator, hence exponential mixing. In particular, our abstract result implies exponential convergence to the unique stationary measure for the class of parabolic PDEs with polynomial nonlinearities considered in Section~4 of~\cite{KNS20.0}.

Applying the refined abstract criterion, we derive the uniqueness and exponential mixing of the stationary measure for the process $(u_n, y_n)$ by showing that the system~\eqref{E:1.1},~\eqref{E:1.2} and its linearization are approximately controllable. 

The ergodicity of the Lagrangian process, combined with the multiplicative ergodic theorem, ensures the existence of the Lyapunov exponents. In particular, the limit \eqref{E:1.5} defining the top Lyapunov exponent~$\la_+$ holds for~$\mu \times \leb$-almost every initial condition $(u_0,x)$. The proof of the positivity of $\la_+$ is much more challenging and also requires taking into account the Jacobian process~$\{D_x\phi^n\}$.

 We emphasize that extending the results of~\cite{BBPS22c} and~\cite{CR24} on the mixing of passive scalars to the setting of the present paper presents significant challenges. The main obstacle is that the associated two-particle process evolves in a non-compact phase space, so the techniques developed here cannot be applied directly to establish its exponential mixing. Addressing this issue falls outside the scope of the present paper; we plan to study it in a future work.

 \medskip
   \noindent 
{\bf A refined Furstenberg criterion.}
As mentioned before, Furstenberg's criterion indicates that the alternative to the positivity of the top Lyapunov exponent is the existence of a certain almost surely invariant structure under the dynamics of the triple $(u_t, \phi^t, D_x\phi^t)$. The existence of this structure cannot be ruled out by applying the approach of~\cite{BBS22} since the strong Feller property is not satisfied. 

When the phase space is compact and the process is uniformly mixing with respect to the total variation metric, a refinement of Furstenberg's criterion is obtained in~\cite{Bou88} and further developed in~\cite{BCG23}. According to that result, if~$\la_+=0$, then there is a continuous deterministic structure that is almost surely invariant under the dynamics. 
However, in the case of a degenerate noise, the Lagrangian process~$(u_n, y_n)$ of the Navier--Stokes system is mixing only with respect to the weaker dual-Lipschitz metric, rendering the results of \cite{Bou88} and~\cite{BCG23}~inapplicable in our setting. Instead, we use an extension of Furstenberg's criterion obtained
in the recent paper~\cite{CR24}. The key observation there is that
a continuous invariant structure can be constructed as soon as the
transition function of the triple matrix process satisfies a
H\"older-type smoothing estimate. The problem thus reduces to
proving such an estimate in our setting. In~\cite{CR24}, where the
forcing is white-in-time, the estimate is derived using the
Malliavin calculus approach of~\cite{HM06}; for the bounded noises
considered here, this tool is not available. We prove the estimate
by a different argument, combining a coupling construction with
the approximate right-inverse method of~\cite{KNS20}, which gives
the H\"older exponent~$1/2$; see Lemma~\ref{E:L3.3}. Let us
mention that any positive exponent would be enough for our
purposes; nevertheless, obtaining the explicit value~$1/2$, rather
than an unspecified exponent as in~\cite{KNS20}, requires some
additional care. Notably, the exponent is universal: it depends
neither on the system under consideration nor on the parameters
entering the construction.

 \medskip   
 
 \noindent 
{\bf Controllability.} We establish both the existence and positivity of the top Lyapunov exponent~$\lambda_+$ by exploiting suitable controllability properties of {\it the triple process}~$(u_t, y_t, D_x\phi^t)$ and its linearization. 
A related controllability property for this triple was used in~\cite{BBS22} to rule out the degenerate scenario in Furstenberg's criterion. 

In the setting of the present paper, however, less information is available about the invariant structure, which we compensate for by employing a stronger controllability property for the triple~$(u_t, y_t,D_x\phi^t)$: we use the fact that the triple can be steered approximately to any target state $(u^\sharp, y^\sharp, A^\sharp)\in X\times \bbt^2\times \Sl_2(\R)$ by~applying controls from the support of the noise. It is worth emphasizing that the compactness of the noise's support does not improve the controllability properties; on the contrary, it makes the problem more degenerate, since the admissible controls are restricted to a compact set involving only a few Fourier modes. The controls are constructed by adapting the
shear-flow method of~\cite{BBS22} (see also~\cite{BBPS22c}), and
the boundedness of the noise creates a genuinely new difficulty
here: one cannot prescribe an arbitrary shear profile, since the
force it generates must belong to the fixed compact
set~$\mathcal K$. At the same time, the shear flow has to produce
the prescribed displacement of the particle or the prescribed transvection of the matrix component, and its amplitude must vanish
at the endpoints of each unit time interval, so that successive
controls can be concatenated. We show that all these requirements
can be met for sufficiently small displacements, and arbitrary
displacements and transvections are then reached by concatenating
finitely many small ones; see Section~\ref{S:5} for the details.

Furthermore, we use the approximate
controllability of the linearization of the triple process.
 Compared with~\cite{KNS20}, where the linearized controllability concerns the velocity equation alone, three new difficulties appear. First, the components $y_t$ and $D_x\phi^t$ are not directly forced: the control enters only the velocity equation and acts on the particle and matrix components indirectly, through the traces $v(t,y)$ and $(D_xv)(t,y)A$. Second, the associated backward system couples the PDE for the dual velocity variable with the ODEs for the dual particle and matrix variables via a distribution-valued term, whose well-posedness requires a separate analysis (see Section~\ref{S:6}); no such coupling is present in~\cite{KNS20}. Third, saturation provides information only in the velocity direction, and the conclusion for the particle and matrix components has to be extracted separately from the coupling in the backward system.

   To formulate our result in a general abstract form, we establish
a new criterion for the positivity of the top Lyapunov exponent for random dynamical systems, building upon
the framework and the results of the papers~\cite{KNS20,KNS20.0}. 
  This criterion shows that the positivity of the top Lyapunov exponent holds provided that one can verify properties of regularity, approximate controllability of the associated nonlinear and linearized systems, and decomposability of the noise.
 We establish the Main Theorem by checking the validity of these hypotheses in the case of the triple matrix process $(u_t, y_t,D_x\phi^t)$ for the Navier--Stokes system. This abstract criterion is of independent interest.

\medbreak
    In summary, the following are the new key ingredients in our proofs:
   \begin{itemize}
   	\item abstract criterion for positivity of the top Lyapunov exponent,
   	\item abstract criterion for exponential mixing without analyticity and without the existence of a unique globally stable equilibrium,
   	\item 
   	H\"older-type estimate for the transition function of the triple matrix process,
   	\item  approximate controllability of the triple process $(u_t, y_t,D_x\phi^t)$ using controls from the compact  support of the noise,
   	\item approximate controllability of the linearization of the triple system using only a few Fourier modes in the velocity equation.  
   \end{itemize}

The paper is organized as follows. Section~\ref{S:2} recalls some basic terminology and facts from the theory of random dynamical systems and presents an abstract version of the refinement of Furstenberg's criterion from~\cite{CR24}.
In Section~\ref{S:3}, we establish an abstract criterion for the positivity of the top Lyapunov exponent and formulate a refined criterion for exponential mixing. The latter is proved in Section~\ref{S:7.2}, where we also derive a H\"older-type estimate for the transition function.
In Section~\ref{S:4}, we formulate and prove a general version of the Main Theorem. In~Sections~\ref{S:5} and \ref{S:6}, we check the validity of the nonlinear and linear approximate controllability properties for the triple~$(u_t, y_t,D_x\phi^t)$ associated with the Navier--Stokes system~\eqref{E:1.1},~\eqref{E:1.2}. 
   Finally, the Appendix contains the proofs of some auxiliary results used throughout the paper.

\subsection*{Acknowledgments}

 The work of D. Zhang is partially supported by 
NSFC (No.12271352, 12322108)
and Shanghai Rising-Star Program 21QA1404500. The work of V. Nersesyan and C. Zhou is partially supported by NSFC (No.~12571156) and NYU Shanghai Boost Fund.

 \subsubsection*{Notation}

 Let $(Z,d_Z)$ be a Polish space endowed with its Borel $\sigma$-algebra~$\mathcal{B}(Z)$. We denote by~$B_Z(a, R)$ and~$\mathring{B}_Z(a, R)$ the closed and open balls in $Z$ of radius $R > 0$ centered at $a \in  Z$. We shall use the following spaces. 
 
   \smallskip
   \noindent 
   $L^\ty (Z)$ is the space of  bounded 
measurable functions $f:Z\to \R $  endowed with the sup-norm $\|\cdot\|_{\infty}$, and $C_b(Z)$ is its subspace consisting of continuous functions.

\smallskip
\noindent
$C_b^{\gamma}(Z),$ $0< \gamma\le 1$ is the space of continuous functions $f:Z\to \mathbb{R}$ such that
$$
|f|_{\gamma}:=\|f\|_{\infty}+\sup_{0<d_{Z}(z_1, z_2)\le 1}\frac {|f(z_1)-f(z_2)|}{d_Z(z_1,z_2)^{\gamma}}<\infty.
$$In the case $\gamma=1$, we write $L_b(Z)$ and $\|f\|_L$ instead of $C_b^{\gamma}(Z)$ and $|f|_{\gamma}$. 
When~$Z$ is compact, we write $C(Z)$, $C^{\gamma}(Z),$ and $L(Z)$.  

\smallskip
\noindent 
$\PP(Z)$ is the set of Borel probability measures on $Z$ endowed with the dual-Lipschitz metric 
$$
\|\mu_1-\mu_2\|_L^*:=\sup_{\|f\|_L\leq1}|\langle f, \mu_1 \rangle- \langle f, \mu_2 \rangle|, \quad \mu_1, \mu_2\in \PP(Z),
$$
where  $\langle f, \mu \rangle :=\int_Z f(z)\mu(\dd z)$ for $f\in C_b(Z)$ and $\mu\in \PP(Z)$.

We shall use the following matrix-related notations. For any integer $d\ge2$, $\M_d(\R)$ is the collection of all $d \times d$ matrices with real entries, $\Gl_d(\R)$ is the~group of invertible matrices, and $\Sl_d(\R)$ is the subgroup of matrices with determinant~$1$. Given a matrix~$A\in \Gl_d(\R)$, we denote by $|A|$ its Euclidean norm and by $A^\TTT$ its transpose. The real $(d-1)$-dimensional projective space~$P^{d-1}$ is the quotient space $\left(\R^d \setminus \{0\}\right)/\sim$ for the   equivalence relation defined by~$x \sim  y$ if there is a non-zero real number $\la $  such that $x = \la y$; the projective space~$P^{d-1}$ is a compact analytic manifold. Note that any matrix~$A\in \Gl_d(\R)$ defines a map from $P^{d-1}$ to $P^{d-1}$ which we denote by the same symbol $A$. Given a measure~$\nu \in \PP(P^{d-1})$ and a matrix~$A\in \Gl_d(\R)$, we write $A_*\nu $ for the pushforward measure of $\nu$ under~$A$. In Section~\ref{S:4.3}, we consider the case $d=2$ and write $P^1$ for the corresponding projective space.

Let 
$$H:=\left\{u\in L^2(\bbt^2, \R^2): \div\,u=0,   \ \int_{\bbt^2}u(x)\dd x=0 \right\}
$$
be the space of divergence-free vector fields on $\bbt^2$ with zero mean value, and let $H^k:=H^k(\mathbb{T}^2,\mathbb{R}^2)$, $k\in\R$, be the usual Sobolev space of order~$k$ endowed with the norm $\|\cdot\|_{H^k}$. For $k\ge0$, we set $V^k:=H^k\cap H$; for $k<0$, $V^k$ is defined as the closure of $H$ in~$H^k$.

 To simplify notation, we shall often use the symbol $\lesssim$ to indicate that an inequality holds up to an unessential multiplicative constant $C$.

Throughout the paper, $(\Omega_0, \mathcal{F}_0, \bbp_0)$ denotes the probability space on which the i.i.d. random variables $\{\eta_k\}_{k\ge1}$ entering the noise are defined, $\ell\in\PP(E)$ stands for their common law, and $\bpP:=\ell^{\otimes\N}$ for the law of the sequence $(\eta_1,\eta_2,\ldots)$ on~$E^{\N}$.

\section{Background on random dynamical systems}\label{S:2}

In this section, we briefly recall some basic concepts and results on random dynamical systems (RDS) relevant to this paper. Given the forthcoming application to the 2D Navier--Stokes system with bounded noise, we only focus here on RDS with independent increments in a compact metric space.
For more details and proofs, we refer the reader to the books \cite{Kif86,Arn98,KS12}.

\subsection{Definitions}

 Let $(\Omega, \mathcal{F},\pP)$ be a probability space, $Z$ be a compact metric space endowed with its Borel $\sigma$-algebra $\mathcal{B}(Z)$, and $\varphi: \Omega\times Z\to Z$ be a measurable mapping such that $\varphi_\omega:Z\rightarrow Z$ is continuous for $\pP$-a.e. $\omega \in \Omega$. We consider an~RDS $\Phi=(\bOmega, \bFF,\bpP, \theta^n, \varphi^n)$ defined as follows.
 \begin{itemize}
 	\item[$\bullet$] $(\bOmega, \bFF,\bpP):=(\Omega, \mathcal{F},\pP)^\N$ with elements $\bomega\in \bOmega$ written as~$\bomega=(\omega_1, \omega_2,\ldots)$.
 	\item[$\bullet$] $\theta:\bOmega\to \bOmega$ is the shift $\theta \bomega=(\omega_2, \omega_3,\ldots)$. Note that it is measure-preserving, that is, $\bpP(\theta^{-1}(\Gamma))=\bpP(\Gamma)$ for any~$\Gamma\in \bFF$. We denote by $\theta^n:\bOmega\to \bOmega$  the $n$-fold composition of~$\theta$ with itself and $\theta^0=\Id_\bOmega$. 
 	\item[$\bullet$]  $\varphi^n: \bOmega\times Z\to Z$ are the compositions of~$\varphi$ with itself given by
\begin{equation}\label{E:2.1}
 \varphi^0_\bomega:=\Id_Z,\quad \varphi^n_\bomega:= \varphi_{\omega_n}\circ\cdots\circ\varphi_{\omega_1},\quad n\ge1.
\end{equation}
 \end{itemize}
   Note that the following cocycle property is satisfied for any $n, m\geq 0$:
$$
 \varphi^{n+m}_\bomega = \varphi^n_{\theta^m \bomega} \circ  \varphi^m_\bomega \quad\quad \text{$\bpP$-a.s.}.
$$

A process $\{z_n\}$ is associated with the RDS $\Phi$ via~$z_n=\varphi_\bomega^n(z)$; it is a Markov process with transition function 
$$
P_n(z, \Gamma):=\bpP\{z_n\in \Gamma\}, \quad z\in Z, \,\, \Gamma\in \mathcal{B}(Z)
$$  and Markov semigroups   
\begin{align*}
\PPPP_n&: L^\ty (Z)\rightarrow  L^\ty(Z), \ \  &\PPPP_n f(z)&:=\int_Z P_n(z, \dd y)f(y), \\
\PPPP_n^*&: \PP(Z)\rightarrow \PP(Z), \ \ &\PPPP_n^*\mu(\Gamma)&:=\int_Z P_n(z, \Gamma)\mu(\dd z).
\end{align*}
The continuity of $\varphi_\omega$ implies that $\PPPP_n$ is Feller, that is,  $\PPPP_n$ maps the space~$C(Z)$ into itself. 
Recall that a measure $\mu\in \PP(Z)$ is called {\it stationary} if~$\PPPP_1^*\mu=\mu$. As~$Z$ is compact, there is at least one stationary measure. Given a stationary measure~$\mu$, a set $\Gamma\in \mathcal{B}(Z)$ is said to be $(\PPPP_1, \mu)$-invariant if the equality $\PPPP_1\mathbb{I}_\Gamma=\mathbb{I}_\Gamma$ holds $\mu$-a.s., where $\mathbb{I}_\Gamma$ is the indicator function of $\Gamma$. A~stationary measure $\mu$ is {\it ergodic} if all  $(\PPPP_1, \mu)$-invariant sets have $\mu$-measure zero or one, and $\mu$ is {\it mixing} if there is a sequence of positive numbers~$\gamma_n\to 0$ as $n\to \ty$ such that
$$
	 \left|\PPPP_n f (z)- \langle f, \mu \rangle  \right|\leq \gamma_n 	 \|f\|_L
$$
for any $z\in Z$ and $f\in L(Z)$. A mixing measure is ergodic.

\subsection{MET and Furstenberg's criterion}

Let the probability spaces $(\Omega, \mathcal{F},\pP)$ and $(\bOmega, \bFF,\bpP)$, the compact metric space~$Z$, and the RDS $\Phi=(\bOmega, \bFF,\bpP, \theta^n, \varphi^n)$ be as in the previous subsection. Assume that~$\aA:\Omega\times Z \to \Gl_d(\R)$, $d\ge2$ is a measurable mapping such that~$ \aA_{\omega,\cdot}:Z \to \Gl_d(\R)$ is continuous for $\pP$-a.e.~$\omega\in \Omega$. We define a sequence of mappings $\aA^n:\bOmega\times Z \to \Gl_d(\R)$ as follows:
\begin{equation}\label{E:2.2}
	 \aA^0_{\bomega,z}:=\Id_{\R^d},\quad \aA^n_{\bomega,z}:= \aA_{\omega_n,\varphi^{n-1}_\bomega(z)}\circ\cdots\circ\aA_{\omega_1, z} \in \Gl_d(\R),\quad n\ge1
\end{equation}
 and note that the cocycle property 
  $$
      \aA^{n+m}_{\bomega,z} = \aA^n_{\theta^m\bomega,\varphi_\bomega^m(z)} \circ  \aA^m_{\bomega,z}  \quad\quad \text{$\bpP$-a.s.}
$$  holds for any $n, m\geq 0$ and $z\in Z$. The sequence $\{\aA^n\}$ is called a
{\it linear cocycle} over the RDS $\Phi$ generated by $\aA$.
  The following is a version of the multiplicative ergodic theorem (MET) of Oseledets~\cite{Osc68}; for a proof we refer to the books~\cite{Arn98,Via14} or the paper~\cite{Rag79}. 
\begin{theorem}\label{T:2.1}
Let $\mu\in \PP(Z)$ be an ergodic stationary measure for $\Phi$, and let~$\{\aA^n\}$ be a linear cocycle over $\Phi$ generated by $\aA$ satisfying the following integrability condition:
\begin{equation}\label{E:2.3}
  \E \int_Z\left(\log^+|\aA_{\omega, z}|+ \log^+|(\aA_{\omega, z})^{- 1}|\right) \mu(\dd z)<\infty,
\end{equation}where $\log^+(x):=\max\{0,\log(x)\}$ for $x>0$ and $\E$ is the expectation with respect to $\pP$.
Then there is an integer $r \in \{1,\ldots, d\}$, deterministic real numbers
$$
\lambda_r< \cdots  < \lambda_1,
$$
and for $\bpP\times \mu$-a.e. $(\bomega, z)$, a flag of subspaces
$$
\{0\}=:F_{r+1} \subsetneq F_r(\bomega, z)\subsetneq  \cdots \subsetneq F_2(\bomega, z) \subsetneq F_1:=\R^d
$$ 
such that for any $i\in  \{1,\ldots, r\}$
$$  \lambda_i = \displaystyle \lim_{n\rightarrow \infty}\frac 1n \log |\aA^n_{\bomega, z}v|,  \quad v\in F_i(\bomega, z) \setminus F_{i+1}(\bomega, z). 
$$
Moreover, the mapping $(\bomega, z)\mapsto F_i(\bomega, z)$ is measurable
and $ \dim  F_i(\bomega,z)$ is constant for $\bpP\times \mu$-a.e. $(\bomega,z)$.
\end{theorem}
The numbers $\lambda_i$ are called {\it Lyapunov exponents} and
\begin{align*}
   m_i := \dim  F_i(\bomega,z) - \dim  F_{i+1}(\bomega,z)
\end{align*}are their {\it multiplicities}. In the proof of Theorem~\ref{T:2.1} it is shown that the Lyapunov exponents $\lambda_i$ are the distinct values among the limits obtained by
\begin{align*}
   \chi_i=\displaystyle \lim_{n\rightarrow \infty} \frac1 n \log \sigma_i(\aA^n_{\bomega, z}),
\end{align*}
where $\sigma_i(\aA^n_{\bomega, z})$ is the $i$-th singular value of the matrix $\aA^n_{\bomega, z}$.
  The existence of the limit and the fact that it is constant
for $\bpP\times \mu$--a.e.~$(\bomega, z)$ is derived from Kingman's subadditive ergodic theorem (see \cite{Kin73}). 
The multiplicities $m_i$ satisfy the relation
\begin{align*}
   m_i = \# \left\{1\leq j\leq d: \chi_j = \lambda_i\right\}.
\end{align*}
 From the fact that $\sigma_1(A)=|A|$ and $\sigma_d(A)=|A^{-1}|^{-1}$ for any $A\in \Gl_d(\R)$, it follows that 
\begin{align*}
\la_1&=  \lim_{n\rightarrow \infty} \frac1 n \log|\aA^n_{\bomega, z}|=: \lambda_+, \\
  \lambda_r& =-  \lim_{n\rightarrow \infty} \frac1 n \log|(\aA^n_{\bomega, z})^{-1}|=:\lambda_-. 	
\end{align*}
 Furthermore,
as $|\dett(A)|=\prod_{i=1}^d \sigma_i(A)$,
we have
\begin{align*}
   \lambda_\Sigma := \sum\limits_{i=1}^r m_i \lambda_i
   = \lim \limits_{n\to \infty} \frac{1}{n} \log |\dett(\aA^n_{\bomega, z})|.
\end{align*}Thus, $\lambda_\Sigma =0$ if $\dett(\aA^n_{\bomega, z})=1$. We arrive at the following corollary of Theorem~\ref{T:2.1}.
\begin{corollary}\label{C:2.2}
	In addition to the conditions of Theorem~\ref{T:2.1}, assume that the generator~$\aA$ takes values in $\Sl_d(\R)$. Then $\lambda_\Sigma =0$.
\end{corollary}
Recall that $P^{d-1}$ is the real projective space of~$\R^d$, and  $\PP(P^{d-1})$ is the set of Borel probability measures on $P^{d-1}$ endowed with the weak convergence topology and the associated Borel $\sigma$-algebra. Given a  set~$D \in  \mathcal{B}(Z)$, we shall say that $\{\nu_z\}_{z\in D} \subset \PP(P^{d-1})$ is a  measurable family if the mapping $z\mapsto \nu_z$ is measurable from $D$ to~$\PP(P^{d-1})$. A  family $\{\nu_z\}_{z\in D} \subset \PP(P^{d-1})$ is said to be weakly continuous if the mapping~$z\mapsto \nu_z$ is continuous from~$D$ to~$\PP(P^{d-1})$.

The next theorem is Furstenberg's criterion originally established in~\cite{Fur63} in a particular case of i.i.d. matrices. The version presented here follows from Proposition~2 and Theorem~3 in~\cite{Led86}.
\begin{theorem}\label{T:2.3} Under the conditions of Theorem~\ref{T:2.1} and the assumption that $\lambda_+=\lambda_-$, there is a measurable family  $\{\nu_z\}_{z\in \supp 
\mu} \subset \PP(P^{d-1})$ such~that 
\begin{equation}\label{E:2.4}
\left(\aA^n_{\bomega, z}\right)_* \nu_z=\nu_{\varphi_{\bomega}^nz}
\end{equation}
 for any $n\ge 1$ and $\bpP\times \mu$-a.e. $(\bomega, z)$.
 \end{theorem}
Note that the family  $\{\nu_z\}_{z\in \supp \mu}$ is deterministic, which makes the invariance relation \eqref{E:2.4} a strong rigidity property for the dynamical system. According to Furstenberg's criterion and Corollary~\ref{C:2.2}, this rigidity is incompatible with the chaoticity (i.e., the positivity of the top Lyapunov exponent~$\la_+$) in~the case when~$\aA$ takes values~in~$\Sl_d(\R)$.

 Based only on the measurability information provided in Theorem~\ref{T:2.3}, it is not clear how to rule out the existence of a family $\{\nu_z\}_{z\in \supp \mu}$ satisfying~\eqref{E:2.4}. However, under the additional assumption that the stationary measure~$\mu$ is mixing in the {\it total variation metric}, it is possible to select the family $\{\nu_z\}_{z\in \supp\mu}$ to be weakly continuous; see Section~4 of~\cite{Bou88} and Proposition~2.10 in~\cite{BCG23}. Unfortunately, these results are not applicable to the Lagrangian trajectories of the Navier--Stokes system, as the presence of degenerate noise only guarantees mixing in the weaker dual-Lipschitz metric.  In our context, we apply the following refinement of Furstenberg's criterion, which is an abstract version of Proposition~3.4~in~\cite{CR24}.

 In the next theorem, we assume that the underlying probability
space $(\Omega, \mathcal{F},\pP)$ has a special structure:
$\Omega$ is a compact metric space, endowed with its Borel
$\sigma$-algebra $\mathcal{F}$ and a Borel probability
measure~$\pP$. Furthermore, we introduce the extended transition function
\begin{equation}\label{E:extend}
\wh{P}_1(z, \Gamma):=\mathbb{P}\{\omega:  (\varphi_{\omega}(z), \aA_{\omega, z})\in \Gamma\}, \quad z\in Z, \; \Gamma\in \mathcal{B}(Z\times \Gl_d(\R)).	
\end{equation}

 \begin{theorem} \label{T:2.4}
In addition to the conditions of Theorem~\ref{T:2.1},
   suppose that the generator~$ \aA: \Omega\times Z \to \Gl_d(\R)$ is continuous, and there is a constant $\gamma\in(0,1]$ such that, for any $c>0$, there is $C=C(c)>0$ with 
\begin{equation}\label{E:ineq1}
\left|\int \wh{P}_1(z_1, \dd y) g(y)-\int \wh{P}_1(z_2, \dd y)g(y)\right|\le \left(C\|g\|_{\infty}+c|g|_{\gamma}\right)d_Z(z_1, z_2)^\gamma	
\end{equation}
for any $ g\in C_b^{\gamma}(Z\times \Gl_d(\R))$ and $z_1, z_2\in Z.$
If $\lambda_+=\lambda_-$,
then there is a weakly continuous family    $\{\nu_z\}_{z\in \supp \mu}\subset \PP(P^{d-1})$  such that the relation \eqref{E:2.4} holds  for any $n\ge 1$, any $z\in \supp 
\mu$, and $\bpP$-a.e. $\bomega$.
\end{theorem}
 The proof of this theorem is discussed in Section~\ref{S:7.1}.

\section{Abstract result}\label{S:3}

This section aims to establish a criterion for the positivity of the top Lyapunov exponent for random dynamical systems. This is achieved by extending the sufficient conditions for mixing obtained in \cite{KNS20,KNS20.0} and building upon the results presented in the previous section.

\subsection{Formulation}\label{Subsec-Gener}

Given any integer $d\ge2$, separable Hilbert spaces $\hH$ and $E$, and a smooth compact Riemannian manifold~$\mM$ without boundary, let us denote by $\HH$ the product space $\hH\times \mM$, and assume that the mappings
\begin{align*}
	S&: \HH\times E\rightarrow  \HH, \quad (z, \eta)\mapsto S(z, \eta),\\
	\aA&: \HH\times E\to \Sl_d(\R), \quad (z, \eta)\mapsto \aA_{\eta, z}
\end{align*}
 are continuous. For any $z_0=z\in  \HH$, let us consider random sequences $\{z_n\}\subset \HH$ and $\{\aA^n_z\}\subset \Sl_d(\R)$ defined by
\begin{align}
z_n&=S(z_{n-1}, \eta_n), \quad 
  n\ge1, \label{E:3.1}\\
 \label{E:3.2}
\aA^n_z&=\aA_{\eta_n, z_{n-1}}\circ\aA^{n-1}_z, \ \ n\ge1,
\end{align}
where $\aA_z^0=\Id_{\R^d}$ and
  $\{\eta_n\}$ is a sequence of i.i.d. random variables in~$E$ defined on a probability space $(\Omega_0, \mathcal{F}_0, \bbp_0)$. Let us assume that the common law~$\ell$ of the random variables $\{\eta_n\}$ has compact support denoted by $\mathcal{K}$, and that there is a compact set $\XX\subset \HH$   that is invariant for the system~\eqref{E:3.1}, i.e., $S(\XX\times \mathcal{K})\subset \XX$. In what follows, we consider the restriction of the system~\eqref{E:3.1} to $\XX$. Under the above conditions,  $\{z_n\}$ is a discrete-time Markov process in $\XX$ with  transition function 
$$
P_n(z, \Gamma):=\bbp_0\{z_n\in \Gamma\}, \quad  z\in \XX, \ \Gamma\in \mathcal{B}(\XX) 
$$
and corresponding Markov semigroups $\PPPP_n$ and $\PPPP_n^*$.
For any $z\in \HH $ and sequence $\{\zeta_n\}\subset E$, we set 
$$
S_n(z; \zeta_1,\ldots, \zeta_n):=z_n \quad \text{ and  }\quad \aA^n(z; \zeta_1,\ldots, \zeta_n):= \aA^n_z,
$$ 
where $z_n$ and $\aA^n_z$ are defined recursively by~\eqref{E:3.1} and~\eqref{E:3.2} with $z_0=z$, $\aA_z^0=\Id_{\R^d}$, and~$\eta_k=\zeta_k$, $k=1, \ldots, n$. 
We assume that the following four hypotheses are satisfied.
 \begin{description}
\item[\hypertarget{H1}{(H$_1$)} Regularity.]
 {\sl There exists a Hilbert space~$\vV$ compactly embedded into~$\hH$ such that the mapping~$S$ is twice continuously differentiable from $\HH\times E$ to~$\vV\times \mM$ with bounded derivatives on bounded subsets.
}

\item[\hypertarget{H2}{(H$_2$)} Approximate controllability to a point.]
{\sl There is a point $\hat z\in \XX$ with the following property: for any $\e>0$, there is an integer $m\ge1$, depending only
on~$\e$, such that, for any $z\in \XX$, there are controls $\zeta_1,\dots,\zeta_m\in\mathcal{K}$ with}
$$
d_{\HH}(S_m(z;\zeta_1,\dots,\zeta_m),\hat z)<\e.
$$
\end{description}
For any $z\in \XX$, let~$\mathcal{K}^z$ be the collection of all $\eta\in \KK$ for which the derivative 
$$
 (D_\eta   S)(z,\eta):E\to \hH \times {\mathsf T}_y\mM=:\mathcal{T}
$$ has a dense image in $\mathcal{T}$, where $y$ is the $\mM$-component of $S(z,\eta)$ and ${\mathsf T}_y\mM$ is the tangent space to $\mM$ at $y$. It is easy to see that~$\mathcal{K}^z$ is a Borel subset of~$E$ (see~Section~1.1 in~\cite{KNS20} for the~details). 
\begin{description}
\item[\hypertarget{H3}{(H$_3$)} Approximate controllability of the linearization.]
{\sl The set~$\mathcal{K}^z$ has full $\ell$-measure for any $z\in \XX$.
 }

\item[\hypertarget{H4}{(H$_4$)} Decomposability of the noise.]
{\sl There exists an orthonormal basis $\{\varphi_j\}$ in~$E$ such that the sequence $\{\eta_k\}$ can be represented in the form
$$
\eta_k=\sum_{j=1}^\infty b_j\xi_{jk}\varphi_j,
$$
where $\{b_j\}$ are real numbers with
$\sum_{j=1}^\infty b_j^2<\infty$ and $\{\xi_{jk}\}$ are independent random variables with Lipschitz-continuous densities~$\rho_j$ with respect to the Lebesgue measure on~$\R$. }
\end{description}
The following is a refined version  of Theorem~1.1 in~\cite{KNS20} and Theorem~1.1 in~\cite{KNS20.0}.  Its proof is discussed in the next section.  
\begin{theorem}\label{T:3.1}
Suppose that Hypotheses {\rm \hyperlink{H1}{(H$_1$)}--\hyperlink{H4}{(H$_4$)}} are satisfied. Then the Markov process $\{z_n\}$ has a unique stationary measure $\mu\in \PP(\XX)$, and there are numbers $C>0$ and $\gamma>0$ such that
\begin{equation} \label{E:3.3}
\|\PPPP_n^*\lambda-\mu\|_L^*\le Ce^{-\gamma n}
\end{equation}for all $n\ge0$ and $\lambda\in\PP(\XX)$, 
where $\|\cdot\|_{L}^*$ is the dual-Lipschitz metric on $\XX$.
\end{theorem}

It is easy to see that, under the conditions of this theorem, the support of the stationary measure $\mu$ is invariant for the system \eqref{E:3.1}, i.e.,  
$$
S(\supp \mu\times \mathcal{K})\subset \supp \mu.
$$ In what follows, we will assume that $\XX=\supp \mu$.

To simplify notation, we shall write  ${\cal R}_n((z, \Id_{\mathbb{R}^d}); \zeta_1,\ldots, \zeta_n)$  instead of the couple $(S_n(z; \zeta_1,\ldots, \zeta_n),\aA^n(z; \zeta_1,\ldots, \zeta_n))$ and introduce the space  $\wh{\HH}:=\HH\times \Sl_d(\R)$. To formulate our sufficient condition for positivity of the top Lyapunov exponent of the linear cocycle~\eqref{E:3.2} over the RDS \eqref{E:3.1},
we~additionally assume that the following three hypotheses are satisfied. 
\begin{description}
	\item[\hypertarget{H5}{(H$_5$)} Regularity.]
  {\sl The mapping $\mathcal{A}$ is twice continuously differentiable from $\mathcal{H}\times E$ to~$\Sl_d(\R)$ with bounded derivatives on bounded subsets.}
\end{description}

\begin{description}
	\item[\hypertarget{H6}{(H$_6$)} Approximate controllability of the extended system.]
  {\sl There is a non-empty open set~$U\subset\Sl_d(\R)$ and points $z_0, z^\sharp\in \XX$ with the following property: for any~$\varepsilon>0$ and any target matrix $A^\sharp\in U$, there is an integer $m\ge 1$ and controls $\zeta_1,\ldots,\zeta_m\in \mathcal{K}$ satisfying}
$$
d_{\wh{\HH}}\left({\cal R}_m((z_0, \Id_{\mathbb{R}^d}); \zeta_1, \ldots, \zeta_m),(z^\sharp, A^\sharp)\right)<\varepsilon.
$$ 
\end{description}
For any $z\in \XX$, let~$\widehat{\mathcal{K}}^z$ be the collection of all $\eta\in \KK$ for which the derivative 
$$
 (D_\eta   {\cal R}_1)(z;\eta):E\to \hH \times {\mathsf T}_y\mM\times  {\mathsf T}_{\aA_{\eta, z}}\Sl_d(\R)=:\widehat{\mathcal{T}}
$$ has a dense image in $\widehat{\mathcal{T}}$. Again, $\widehat{\mathcal{K}}^z$ is a Borel subset of~$E$. 

\begin{description}
\item[\hypertarget{H7}{(H$_7$)} Linearization of the extended system.]
{\sl The set~$\widehat{\mathcal{K}}^z$ has full $\ell$-mea\-sure for any $z\in \XX$.
 }
\end{description}
As $\widehat{\mathcal{K}}^z$ is contained in $\mathcal{K}^z$, Hypothesis \hyperlink{H7}{\rm(H$_7$)} is stronger than \hyperlink{H3}{\rm(H$_3$)}.

The main result of this section is the following theorem; its proof is given in the next subsection. 
\begin{theorem}\label{T:3.2}
Suppose that Hypotheses {\rm \hyperlink{H1}{(H$_1$)}--\hyperlink{H7}{(H$_7$)}} are satisfied.
Then there exists a deterministic constant $\lambda_+>0$
such that 
\begin{align}\label{E:3.4}
   \lambda_+=\displaystyle \lim_{n\rightarrow \infty}\frac 1n  \log|\aA^n_{z}|\quad \quad \text{for  $\pP_0\times \mu$-a.e. $(\omega, z)$}.
\end{align}
\end{theorem}

\subsection{Proof of Theorem~\ref{T:3.2}}\label{S:3.2}

The proof of Theorem~\ref{T:3.2} is divided into two steps: first, we show that the limit $\lambda_+$ in \eqref{E:3.4} exists and is constant for $\pP_0\times \mu$-a.e. $(\omega,z)$, then we prove that~$\lambda_+>0$. To achieve this, we represent the systems \eqref{E:3.1} and~\eqref{E:3.2} in the form discussed in Section~\ref{S:2} and apply the MET together with the improved version of Furstenberg's criterion.

\paragraph{\bf Step 1.} The system \eqref{E:3.1} defines an RDS as in Section~\ref{S:2} by choosing 
\begin{itemize}
	\item[$\bullet$] the probability space
$(\Omega, \mathcal{F},\pP):=(\mathcal{K}, \mathcal{B}(\mathcal{K}), \ell)$,
	\item[$\bullet$] the compact metric space $Z:=\XX$,
	\item[$\bullet$] the mapping $\varphi: \Omega\times Z\to Z$, $(\omega,z)\mapsto S(z,\omega)=:\varphi_\omega(z)$ and its compositions $\varphi^n_\bomega$ defined by \eqref{E:2.1},
	\item[$\bullet$] the product space $(\bOmega, \bFF, \bpP):=(\mathcal{K}, \mathcal{B}(\mathcal{K}), \ell)^{\mathbb{N}}$ with elements written as $\bomega=(\omega_1,\omega_2,\ldots)$ and the shift $\theta\bomega=(\omega_{2}, \omega_{3},\ldots)$, 
	\item[$\bullet$] the generator $\aA:\Omega\times Z \to \Sl_d(\R)$ and its compositions $\aA^n_{\bomega, z}$ defined by~\eqref{E:2.2}.

	\end{itemize}
By Theorem~\ref{T:3.1}, this RDS has a unique stationary measure $\mu\in \PP(\XX)$ which is mixing in the sense that \eqref{E:3.3} holds. Furthermore, the integrability condition~\eqref{E:2.3} is satisfied as $\aA:  \KK\times \XX\rightarrow \Sl_d(\R)$ is continuous and $\XX$ and~$\KK$ are compact. Thus, by Theorem~\ref{T:2.1},
there exists a deterministic constant $\lambda_+$
such that
$$
   \lambda_+=\displaystyle \lim_{n\rightarrow \infty}\frac 1n  \log|\aA^n_{\bomega, z}| \quad \quad \text{for $\bpP\times \mu$-a.e. $(\bomega, z)$}.
$$ To deduce \eqref{E:3.4}, define the measurable
mapping $\iota:\Omega_0\to\bOmega$ by
$\iota(\omega):=(\eta_1(\omega),\eta_2(\omega),\ldots)$. Since the
random variables $\{\eta_n\}$ are independent with common
law~$\ell$, we have $\iota_*\pP_0=\ell^{\otimes\N}=\bpP$.
Moreover, $\aA_z^n(\omega)=\aA^n_{\iota(\omega),z}$ for all~$\omega\in\Omega_0$ by construction, so \eqref{E:3.4} follows from the above limit by pulling back the
$\bpP\times\mu$-full measure set of its validity under the mapping
$(\omega,z)\mapsto(\iota(\omega),z)$. As $\aA$ is $\Sl_d(\R)$-valued, we have $\la_+\ge0$
by Corollary~\ref{C:2.2}.

\paragraph{\bf Step 2.} To prove the positivity of the exponent $\la_+$, we use Theorem~\ref{T:2.4}. The conditions of that theorem are fulfilled due to the following result whose proof is deferred to Section~\ref{SS:7.2.2}.  
\begin{lemma}\label{E:L3.3}
   Under Hypotheses {\rm \hyperlink{H1}{(H$_1$)}}, {\rm \hyperlink{H4}{(H$_4$)}}, {\rm \hyperlink{H5}{(H$_5$)}}, and {\rm \hyperlink{H7}{(H$_7$)}}, for any $c>0$ there is a constant $C:=C(c)>0$ such that 
\begin{equation}\label{E:ineq2}
\left|\int \wh{P}_1(z_1, \dd y) g(y)-\int \wh{P}_1(z_2, \dd y)g(y)\right|\le \left(C\|g\|_{\infty}+c|g|_{\frac12}\right)d_\XX(z_1, z_2)^\frac12	
\end{equation}
for any $ g\in C_b^{\frac12}(\XX \times  \Sl_d(\R))$ and $z_1, z_2\in \XX$, where  $\wh{P}_1$ is the extended transition function defined by \eqref{E:extend}.
\end{lemma} 

 Arguing by contradiction, let us assume that $\lambda_+=0$. From Corollary~\ref{C:2.2} it follows that $\lambda_+=\lambda_-=0$, which, combined with Theorem~\ref{T:2.4}, implies the existence of a weakly continuous family of measures   $\{\nu_z\}_{z\in \XX}\subset \PP(P^{d-1})$ satisfying the equality \eqref{E:2.4} for any $n\ge 1$, any $z\in 
\XX$, and $\bpP$-a.e. $\bomega$.
 
Let the open set $U\subset\Sl_d(\R)$ and the points $z_0, z^\sharp\in \XX$ be as in Hypothesis~\hyperlink{H6}{\rm(H$_6$)}.
For any given measures $\nu, \nu' \in \PP(P^{d-1})$, the~set
$$
\{ A\in\Sl_d(\R): A_*\nu=\nu'\} 
$$
has an empty interior in $\Sl_d(\R)$. Therefore, there exists $A^\sharp\in U$ such that
$$ 
(A^\sharp)_*\nu_{z_0}\neq \nu_{z^\sharp}.
$$Let us denote
\begin{align}  \label{E:3.5}
    \|(A^\sharp)_*\nu_{z_0}- \nu_{z^\sharp}\|_L^* =: \ve>0,
\end{align}
where $\|\cdot\|_L^*$ is the dual-Lipschitz metric on $\PP(P^{d-1})$.  By Theorem~\ref{T:2.4},
we have the continuity of the maps
 $z\mapsto \nu_z$, $\XX \to  \PP(P^{d-1})$ and 
   $A \mapsto A_* \nu_{z_0}$, $\Sl_d(\R) \to  \PP(P^{d-1})$. 
Hence, there exists $\delta>0$
such that
 \begin{align}  \label{E:3.6}
   \|\nu_{z}-\nu_{z^\sharp}\|_L^* +
   \|(\wh A)_*\nu_{z_0}- (A^\sharp)_* \nu_{z_0}\|_L^*< \frac{\ve}{2} 
\end{align} for any $z\in B_{\XX}(z^\sharp, \delta)$ and $\wh A\in B_{\Sl_d(\R)}(A^\sharp, \delta)$.
Now, according to Hypothesis~{\rm \hyperlink{H6}{(H$_6$)}}, there is an integer $m\ge 1$ and controls $\zeta_1,\ldots,\zeta_m\in \mathcal{K}$ verifying   
\begin{equation}\label{E:3.7}
	 d_{\wh \HH}\left({\cal R}_m((z_0, \Id_{\mathbb{R}^d}); \zeta_1, \ldots, \zeta_m) , (z^\sharp, A^\sharp)\right) <\frac{\delta}{2}.
\end{equation}
 For any $\delta'>0$, let the event ${\bf\Gamma}_{\delta'}\in \bFF$ be defined by 
$$
{\bf\Gamma}_{\delta'}=\prod_{l=1}^m B_\KK(\zeta_l,\delta') \times \KK\times \KK\times\ldots.
$$ By the continuity of~$S$ and $\aA$, we have
 \begin{align*}
   & d_{\wh \HH}\left((\varphi_{\bomega}^m{z_0}, \aA^m_{\bomega, z_0}),{\cal R}_m((z_0, \Id_{\mathbb{R}^d}); \zeta_1, \ldots, \zeta_m)\right)<\frac{\delta}{2} 
\end{align*} for sufficiently small $\delta'\in(0,\delta)$ and any $\bomega\in{\bf\Gamma}_{\delta'}$.
Therefore, in view of \eqref{E:3.7},  
$$
 d_{\wh \HH}\left((\varphi_{\bomega}^mz_0, \aA^m_{\bomega, z_0}),(z^\sharp, A^\sharp)\right) < \delta, \quad \bomega\in {\bf\Gamma}_{\delta'}.
$$
Taking into account \eqref{E:3.6}, we obtain
$$
    \| \nu_{\varphi_{\bomega}^mz_0}-\nu_{z^\sharp}\|_L^* +
   \|(\aA^m_{\bomega, z_0})_* \nu_{z_0}- (A^\sharp)_*\nu_{z_0}\|_L^*\leq \frac{\ve}{2}.
$$
Combining this with~\eqref{E:3.5}, we derive
$$
   \|(\aA^m_{\bomega, z_0})_*\nu_{z_0}-\nu_{\varphi_{\bomega}^mz_0}\|_L^*
   > \frac{\ve}{2} >0 \quad \text{for } \bomega\in {\bf\Gamma}_{\delta'},
$$
which contradicts the equality \eqref{E:2.4} for $z=z_0$, since ${\bf\Gamma}_{\delta'}$ has positive $\bpP$-probability. The proof of Theorem~\ref{T:3.2} is complete.

\section{Mixing via controllability}
\label{S:7.2}

In this section, we prove Theorem~\ref{T:3.1} and Lemma~\ref{E:L3.3}. As a preliminary step, we first extend the construction of an approximate right inverse of linear operators with dense image from Section~2.5 of~\cite{KNS20} to the non-analytic setting.
 In the subsequent subsection, we use this construction to define a transformation that possesses appropriate squeezing and measure-transformation properties.

 \subsection{Approximate right inverses without analyticity}
 
   Let us recall the framework of Section~2.5 in~\cite{KNS20} and give a proof of the main result obtained there under the hypotheses of the current paper, i.e., without\footnote{To simplify notation and the presentation, in this and next subsections, we will not care about the manifold component and assume that the phase space is a separable Hilbert~space.} analyticity. Throughout, we assume that
   \begin{itemize}
      	 	\item $\XX$ is a compact metric space,
   	\item  $E, F$, and~$\TT$ are separable Hilbert spaces and $\VV$ is a Hilbert space compactly embedded into~$\TT$,
   	 	 	\item $\ell\in\PP(E)$ is a measure with a compact support~$\KK$,
    	 	 	 	 	\item $\{\psi_j\}$ is an orthonormal basis in~$F$ and~$F_M:=\text{span}\{\psi_1,\dots,\psi_M\}$ for~$M\ge1$.
   \end{itemize}
 The following is a version of Theorem~2.8 in~\cite{KNS20} without analyticity assumption for the linear mapping $\mathbb{A}$.
\begin{theorem} \label{p2.5}
Let $\mathbb{A}:\XX\times E\to\LL(F,\TT)$ be a continuous mapping such that, for any $z\in \XX$, there is a Borel set $\KK^z\subset\KK$ verifying 
\begin{enumerate}
	\item[\rm (a)] $\ell (\KK^z)=1$,
	\item[\rm (b)] the image of $\mathbb{A}(z,\eta)$ is dense in~$\TT$ for any $\eta\in\KK^z$. 
\end{enumerate} 
 Then,  for any $\e=(\e_1,\e_2)\in(0,1)^2$, there is a continuous function 
 $$
 \FFFF_\e(\cdot,\cdot):\XX\times E\to \R_+
 $$ such that $\eta\mapsto \FFFF_\e(z,\eta)$ is analytic for any $z\in \XX$, and there is an integer $M_\e\ge1$ and positive constants $\nu_{\e_2}$ and $C_\e$ such that the following properties are satisfied.
\begin{description}
\item[$\bullet$]We have
\begin{equation}
	\ell(\KK_{\e}^{z})\ge 1-\e_1\quad	\mbox{for $z\in \XX$},
\label{E:7.5}
\end{equation}where
\begin{equation} \label{E:7.6}
\KK_\e^{z}:=\left\{\eta\in\KK: \FFFF_\e(z,\eta)\le\nu_{\e_2}\right\}.
\end{equation}
\item[$\bullet$]
Let $\DD_\e\subset  \XX\times \KK$ be the compact subset defined by
\begin{equation} \label{E:7.7}
\DD_\e:=\left\{(z,\eta)\in \XX\times \KK: \FFFF_\e(z,\eta)\le2\nu_{\e_2}\right\}.
\end{equation}
There is a continuous mapping $R_\e:\DD_{\e}\to\LL(\TT,F)$ such that
\begin{gather}
\textup{Im} \left(R_\e(z,\eta)\right)\subset F_{M_\e}, \quad
\|R_\e(z,\eta)\|_{\LL(\TT,F)}\le C_\e,
\quad (z,\eta)\in \DD_\e,
\label{E:7.8}\\
\|\mathbb{A}(z,\eta)R_\e(z,\eta)f-f\|_\TT\le \e_2\|f\|_\VV,
\quad (z,\eta)\in \DD_\e,\, f\in \VV.
\label{E:7.9}
\end{gather}
\end{description}
\end{theorem}
\begin{proof} 
For any number $\gamma>0$ and integer $M\ge1$, let us denote 
\begin{align}
		G(z,\eta)&:=\mathbb{A}(z,\eta)\mathbb{A}^*(z,\eta),\nonumber\\
		R_\gamma(z,\eta)&:=\mathbb{A}^*(z,\eta)\left(G(z,\eta)+\gamma I\right)^{-1},\label{EE:EER}\\
		R_{M,\gamma}(z,\eta)&:={\mathsf P}_MR_\gamma(z,\eta),\nonumber
\end{align}
where ${\mathsf P}_M:F\to F$ is the orthogonal projection onto~$F_M$. Note that $\left(G(z,\eta)+\gamma I\right)^{-1}$ is well defined, since $G(z,\eta)\ge0$.

We will show that \eqref{E:7.8} and \eqref{E:7.9} are satisfied for $R_{M,\gamma}$ for some choice of~$\gamma$, $M$, and $C$. First, note that \eqref{E:7.8} holds for any $\gamma$ and $M$. Indeed, by definition, $\textup{Im} \left(R_{M,\gamma}(z,\eta)\right)\subset F_M$ and we have
\begin{align*}
\|R_{M,\gamma}(z,\eta)\|_{\LL(\TT,F)}&\le \|\mathbb{A}^*(z,\eta)\|_{\LL(\TT,F)} \|\left(G(z,\eta)+\gamma I\right)^{-1}\|_{\LL(\TT)}\\
&\le \gamma^{-1} \sup_{(z,\eta)\in \XX\times \KK}\|\mathbb{A}(z,\eta)\|_{\LL(F,\TT)} =:C<\ty,
\end{align*}
 where we used the bound $\|\left(G(z,\eta)+\gamma I\right)^{-1}\|_{\LL(\TT)}\le \gamma^{-1}$
 and the equality $\|\mathbb{A}^*(z,\eta)\|_{\LL(\TT,F)}=\|\mathbb{A}(z,\eta)\|_{\LL(F,\TT)}$.

Let us take any $\e=(\e_1,\e_2)\in(0,1)^2$ and assume that we have constructed a continuous function $\FFFF_\e:\XX\times E\to \R_+$ such that $\eta\mapsto \FFFF_\e(z,\eta)$ is analytic for any $z\in \XX$, there are positive numbers $\nu_{\e_2}$ and $\gamma_\e$ such that \eqref{E:7.5} holds with the set $\KK_\e^{z}$ in \eqref{E:7.6}, and the following inequality is verified
\begin{equation}\label{E:10}
\sup_{(z,\eta)\in \DD_\e}\|\mathbb{A}(z,\eta)R_{\gamma_\e}(z,\eta)f-f\|_\TT<\e_2, \quad f\in  B_\VV(0,1)
\end{equation}with the set $\DD_\e$ in \eqref{E:7.7}. Using the compactness of $\KK$, $ \DD_\e$, and $B_\VV(0,1)$ and the convergence of ${\mathsf P}_M$ to $I$ as $M\to \ty$, it is easy to see that 
$$
\sup_{(z,\eta)\in \DD_\e}\|\mathbb{A}(z,\eta)R_{M_\e,\gamma_\e}(z,\eta)f-f\|_\TT<\e_2, \quad f\in  B_\VV(0,1)
$$ for sufficiently large $M_\e\ge1$; see the proof of Theorem~2.8 in~\cite{KNS20} for~details. 

Thus, it remains to construct $\FFFF_\e$, $\nu_{\e_2}$, and $\gamma_\e$.  
 To this end, let $ \hat\nu_{\e_2}:=\e_2^2/16$, and
 let us define a continuous function $\hat\FFFF_\gamma:\XX\times E\to \R_+$~by
$$
\hat\FFFF_\gamma(z,\eta)=\sum_{j=1}^N \|\mathbb{A}(z,\eta)R_{\gamma}(z,\eta) f_j-f_j\|_\TT^2,
$$where 
\begin{equation}\label{E:7.11}
	\{f_j: j=1,\ldots, N\} \,\, \text{is an $\e_2/4$-net for the compact set $B_\VV(0,1)\subset \TT$.}
\end{equation} By Lemma~3.1 in~\cite{KNS20}, there is $\gamma_{\e}>0$ such that
\begin{equation}
	\ell(\hat\KK_{\e}^{z})\ge 1-\e_1\quad	\mbox{for $z\in \XX$},
\label{E:7.12}
\end{equation}where
$$
\hat\KK_\e^{z}:=\left\{\eta\in\KK: \hat\FFFF_{\gamma_\e}(z,\eta)\le\hat\nu_{\e_2}\right\}.
$$   On the other hand, for any $f\in B_\VV(0,1)$, we have
    \begin{align}
    \|\mathbb{A}(z,\eta)R_{\gamma_\e}(z,\eta)f-f\|_\TT &\le \underset{1 \le j \le N}{\text{min}}\left(\|\mathbb{A}(z,\eta)R_{\gamma_\e}(z,\eta)(f-f_j)\|_\TT+\|f-f_j\|_\TT\right)\nonumber\\&\quad
    +\hat\FFFF_{\gamma_\e}(z,\eta)^\frac{1}{2}	\nonumber\\&\le \frac{\e_2}{2}+\hat\FFFF_{\gamma_\e}(z,\eta)^\frac{1}{2},\label{E:7.13}
    \end{align}
    where we used \eqref{E:7.11} and the fact that the norm of the operator $G(G+\gamma_\e I)^{-1}$ is bounded by $1$. As the map $\eta\mapsto\hat\FFFF_\gamma(z,\eta)$ is not necessarily analytic, we introduce analytic approximations.

  Let $\mathbb{A}_n:\XX\times E\to\LL(F,\TT)$ be a sequence of continuous maps such that 
    \begin{itemize}
  	\item $\eta\mapsto \mathbb{A}_n(z,\eta)$ is analytic for any $z\in \XX$ and $n\ge1$, and the derivative $(z,\eta)\mapsto (D_\eta\mathbb{A}_n)(z,\eta)$ is continuous on $\XX\times E$,
  	\item $\sup_{(z,\eta)\in \XX\times \KK} \|\mathbb{A}_n(z,\eta)-\mathbb{A}(z,\eta)\|_{\LL(F,\TT)}\to 0 $ as $n\to \infty.$ 
  	  \end{itemize}Let us set
  $$
  	  		  	  	\FFFF_{n,\e}(z,\eta):=\sum_{j=1}^N \|\mathbb{A}_n(z,\eta)R_{n,\gamma_\e}(z,\eta) f_j-f_j\|_\TT^2,
$$
  where 
  \begin{align*}
  	R_{n,\gamma_\e}(z,\eta)&:=  \mathbb{A}_n^*(z,\eta)\left(G_n(z,\eta)+\gamma_\e  I\right)^{-1},\\
  	G_n(z,\eta)&:=\mathbb{A}_n(z,\eta)\mathbb{A}_n^*(z,\eta).
  \end{align*}
Then, $\FFFF_{n,\e}(z,\eta):\XX\times E\to \R_+$
    is  a continuous function and $\eta\mapsto \FFFF_{n,\e}(z,\eta)$ is analytic for any $z\in \XX$. Our goal is to show that \eqref{E:7.5} and \eqref{E:10} hold with~$\FFFF_{n,\e}$  for some number $\nu_{\e_2}>0$ and integer $n\ge1$.

        Let us take $\delta:=\e^2_2/32$  and choose $n\ge1$ so large that
\begin{equation}\label{E:7.14}
	    \sup_{(z,\eta)\in \XX\times \KK } |\hat\FFFF_{\gamma_\e}(z,\eta)-\FFFF_{n,\e}(z,\eta)|<\delta.
\end{equation}
    Clearly,
    $$
    \left\{\eta\in\KK: \hat\FFFF_{\gamma_\e}(z,\eta)\le\hat\nu_{\e_2}\right\}
    \subset
    \left\{\eta\in\KK: \FFFF_{n,\e}(z,\eta)\le\hat\nu_{\e_2}+\delta\right\}.
    $$ We now fix such an $n$ and set $\FFFF_\e:=\FFFF_{n,\e}$.
   Then, in view of \eqref{E:7.12}, we have that \eqref{E:7.5} is satisfied for $\KK_\e^{z}$ as in \eqref{E:7.6} with $\nu_{\e_2}:=\hat\nu_{\e_2}+\delta$. From \eqref{E:7.13} and \eqref{E:7.14} it follows that, if $\FFFF_{n,\e}(z,\eta) \le 2 \nu_{\e_2}$, then
\begin{align*}
    \|\mathbb{A}(z,\eta)R_{\gamma_\e}(z,\eta)f-f\|_\TT & \le \frac{\e_2}{2}+\hat\FFFF_{\gamma_\e}(z,\eta)^\frac{1}{2} \\&\le \frac{\e_2}{2}+ \left(\FFFF_{n,\e}(z,\eta) +\delta \right)^\frac{1}{2}\\&\le \frac{\e_2}{2}+ \left( 2 \hat \nu_{\e_2}  +3\delta \right)^\frac{1}{2}<\e_2, 	
\end{align*}
 which completes the proof of the theorem.
 \end{proof}

\subsection{Construction of a squeezing transformation}

Throughout the rest of this section, we remain within the framework of Section~\ref{S:3}. The following result will play a key role in the proofs of Theorem~\ref{T:3.1} and Lemma~\ref{E:L3.3}; it is a version of Proposition~2.3 in~\cite{KNS20} with specified exponent~1/2 in \eqref{E:aa3}. Let $\wh{\mathcal{K}}^{z}$ be the set in Hypothesis {\rm \hyperlink{H7}{(H$_7$)}}, and let 
 $$
 D_{\delta}:=\left\{(z,z')\in\mathcal{X}\times \mathcal{H}: \|z-z'\|\le \delta\right\}.
 $$ 
 \begin{proposition}\label{PP:7.4}
Under Hypotheses {\rm \hyperlink{H1}{(H$_1$)}}, {\rm \hyperlink{H4}{(H$_4$)}}, {\rm \hyperlink{H5}{(H$_5$)}}, and {\rm \hyperlink{H7}{(H$_7$)}}, for any $\sigma, \theta\in (0,1)$, there are constants $C>0$ and~$\delta>0$, a Borel-measurable mapping $\Phi: \mathcal{X}\times \mathcal{H}\times E\to E$, and a family of Borel subsets $\{\wh{\mathcal{K}}^{z}_{\sigma, \theta}\subset \wh{\mathcal{K}}^{z}\}_{z\in \mathcal{X}}$ such that $\Phi^{z,z'}(\eta)=0$ if $\eta\notin \wh{\mathcal{K}}^{z}_{\sigma, \theta}$ or $z=z'$, and the following inequalities hold:
\begin{gather}
\ell(\wh{\mathcal{K}}^{z}_{\sigma, \theta})\ge 1-\sigma, \label{E:aa1}\\
\|\ell-\Psi_*^{z, z'}(\ell)\|_{\textup{var}}\le C\|z-z'\|^{\frac12}, \label{E:aa3} \\
\|{\cal R}_1(z; \eta)-{\cal R}_1(z';  \Psi^{z,z'}(\eta))\|\le \theta\|z-z'\|\label{E:aa2}
\end{gather} for any $\eta\in \wh{\mathcal{K}}^{z}_{\sigma, \theta}$ and 
$(z,z')\in D_{\delta}$, where $\Psi^{z,z'}(\eta):=\eta+\Phi^{z, z'}(\eta)$. 
\end{proposition}

	Let us explain the modification that needs to be done in the proof of Proposition~2.3 in~\cite{KNS20} in order to obtain power 1/2 in \eqref{E:aa3}. The proof of the proposition relies on an application of a measure transformation theorem under Lipschitz maps, which is Theorem~2.4 in~\cite{KNS20}. The goal is to construct a Borel-measurable mapping $\Phi: \mathcal{X}\times \mathcal{H}\times E\to E$ and Borel subsets $\{\wh{\mathcal{K}}^{z}_{\sigma, \theta}\subset \wh{\mathcal{K}}^{z}\}_{z\in \mathcal{X}}$ that satisfy the above inequality \eqref{E:aa1}, as well as   properties (a) and (b) in Theorem~2.4 in~\cite{KNS20} with power~$\gamma=1$ in~(b) (we do not recall the formulations of these properties here). Application of Theorem~2.4 proves \eqref{E:aa3} with power $\gamma/(1+\gamma)=1/2$ and \eqref{E:aa2}.

The construction of $\Phi$ and $\wh{\mathcal{K}}^{z}_{\sigma, \theta}$ is based on Theorem~2.8 in~\cite{KNS20}, which in our setting is presented in Theorem~\ref{p2.5}. We will focus only on the place that allows us to justify that $\gamma=1$ in~(b) of Theorem~2.4. Theorem~\ref{p2.5} is applied with $\e_1=\e_2=\e$ for the mapping $\mathbb{A}(z,\eta)=D_\eta {\cal R}_1(z,\eta)$ and the space~$F=E$, and we denote by $\KK_{\e}^{z}, \DD_\e, \FFFF_\e(z,\eta), R_\e(z,\eta), \nu_\e$ the corresponding objects from that theorem.  Recall that $\KK_\e^{z}$, $\DD_\e$, and $R_\e$ have the form
\begin{align}
	\KK_\e^{z}&=\{\eta\in\KK:\FFFF_\e(z,\eta)\le\nu_\e\}
	\quad\mbox{for $z\in \XX$},\label{AAA1}\\
	\DD_\e&=\{(z,\eta)\in \XX\times \KK:\FFFF_\e(z,\eta)\le2\nu_\e\}, 
	\label{AAA2}\\
	R_\e(z,\eta)&={\mathsf P}_MR_{\gamma}(z,\eta)
	\quad\mbox{for $(z,\eta)\in\DD_\e$},\label{AAA3} 
\end{align}
where $R_\gamma$ is defined in~\eqref{EE:EER}, ${\mathsf P}_M$ is the orthogonal projection onto some $M$-dimensional subspace $E_M$ of $E$, and
 the number $\gamma=\gamma(\e)>0$ and the integer~$M=M(\e)\ge1$ are chosen appropriately. The mapping~$\Phi$ is defined~by 
\begin{equation}\label{EEE:EEW}
\Phi^{z,z'}(\eta):=-R_\e(z,\eta)(D_z{\cal R}_1)(z,\eta)(z'-z). 	
\end{equation}
The definition of $\wh{\mathcal{K}}^{z}_{\sigma, \theta}$ is more subtle. Let~$\DD_\e'$ and $\KK'$ be the projections of~$\DD_\e$ and $\KK$ to $\XX\times E_{M}^\bot$ and $E_{M}$, that~is,   
\begin{gather*}
	\DD_\e':=\{(z,w)\in \XX\times E_{M}^\bot:\mbox{ $\exists$ $v\in E_{M}$ s.t. $w+v\in\KK_\e^z$}\},\\ 
\KK':=\{v\in E_M:\mbox{ $\exists$ $w\in E_{M}^\bot$ s.t. $w+v\in\KK$}\}. 
\end{gather*}
These sets are compact, as projections of compact sets. Let $O_M\subset E_M$ be a bounded convex open set containing $\KK'$. 
The following is a modification of Lemma~2.9 in \cite{KNS20}.  
\begin{lemma} \label{l3.3}
There are disjoint sets $\DD_1',\dots,\DD_m'$ with
$
\DD_\e'= \cup_{l=1}^m \DD_l'
$
and  numbers $\nu_1,\ldots,\nu_m\in(\nu_\e,3\nu_\e/2)$ such that, for any $(z,w)\in\overline{\DD_l'}$, zero is a regular value for the~function 
$$
v\mapsto\FFFF_\e(z,w+v)-\nu_l, \quad \overline{O_M}\to \R, \quad 1\le l \le m.
$$  
\end{lemma}
 Now, the set $\wh{\mathcal{K}}^{z}_{\sigma, \theta}$ is defined by
\begin{equation} \label{3.021}
\wh{\mathcal{K}}^{z}_{\sigma, \theta}:=\bigcup_{l=1}^m
\bigl\{\eta\in \KK:\eta=v+w, \,\, (z,w)\in\DD_l', \,\, \FFFF_\e(z,w+v)\le\nu_l\bigr\}	
\end{equation}
and $\Phi^{z,z'}$ in~\eqref{EEE:EEW} is modified to be zero for $\eta\notin \wh{\mathcal{K}}^{z}_{\sigma, \theta}$.

After this, the justification that, for sufficiently small $\e>0$, the constructed $\Phi$ and $\wh{\mathcal{K}}^{z}_{\sigma, \theta}$  satisfy properties (a) and (b) in Theorem~2.4 in~\cite{KNS20} with power~$\gamma=1$ in~(b), is exactly the same as in~\cite{KNS20}, with Corollary~3.3 in~\cite{KNS20} replaced by Corollary~\ref{T:4.3} below; the continuity hypothesis of that corollary holds, since $D_\eta \mathbb{A}_n$, and hence $D_v\FFFF_{n,\e}$, is continuous by construction.

  \begin{proof}[Proof of Lemma~\ref{l3.3}]
	For any $(z,w)\in\DD_\e'$, the map $v\mapsto \FFFF_\e(z,w+v)$ is smooth on the space~$E_{M}$. By Sard's theorem, almost every real number is a regular value for it. Thus, we can choose a number $\nu_{z,w}\in(\nu_\e,3\nu_\e/2)$ such that zero is a regular value of the function $v\mapsto \FFFF_\e(z,w+v)-\nu_{z,w}$. By~continuity of $\FFFF_\e$ and $D_v\FFFF_\e$ with respect to~$(z',w',v)$, there is a non-degenerate closed ball~$B_{z,w}$ centred at~$(z,w)$ such that for all~$(z',w') \in B_{z,w}$, the function $v\mapsto \FFFF_\e(z',w'+v)-\nu_{z,w}$ has no critical point in the compact set~$\overline{O_M}$ at which it vanishes.
	 Since the corresponding open balls cover the compact set~$\DD_\e'$, we can extract a sub-covering $B_{z_j,w_j}$, $j=1,\dots,m$. Finally, we define~$\DD_l':=(\DD_\e'\cap B_{z_l,w_l})\setminus \bigl(\,\cup_{j=1}^{l-1}\DD_j'\bigr)$, where the union is empty for~$l=1$. 
\end{proof}

\paragraph{An auxiliary result.} 
Let~$\XX$ be a compact metric space, let~$E$ be a finite-dimensional space with orthonormal basis~$\{e_j\}_{j=1}^M$, and let~$\OO\subset E$ be a compact convex set.  Let $\FFFF:\XX\times \OO\to \R $ be a continuous function such~that, for any $z\in \XX$, the function $\FFFF(z,\cdot): \OO\to \R $ is smooth with zero being a regular value, and the derivative $(z,\eta)\mapsto(D_\eta \FFFF)(z,\eta)$ is continuous on $\XX\times\OO$.
\begin{lemma} \label{l2.6}
Under the above conditions, 
 there is a constant~$C>0$ such~that
\begin{equation} \label{2.25}
\Leb\bigl(\{\eta\in\OO:|\FFFF(z,\eta)|\le r\}\bigr)\le C\,r
\end{equation} for any $z\in \XX$ and~$r\in[0,1]$.
\end{lemma}
This lemma allows to estimate the measure of a tubular neighborhood of the nodal set of $\FFFF$. That is, for any~$z\in \XX$,~let
$$
\NN(z):=\{\eta\in \OO: \FFFF(z,\eta)=0\}.
$$

\begin{corollary} \label{T:4.3}
Under the conditions of Lemma~\ref{l2.6}, there is $C>0$ such~that 
\begin{equation}\label{4.8}
\Leb\bigl(\{\eta\in\OO:\dist(\eta,\NN(z))\le r\}\bigr)\le C\,r
\end{equation}for any $z\in \XX$ and~$r\in[0,1]$.
\end{corollary}
\begin{proof}
Given the convexity of~$\OO$ and the boundedness of $D_\eta \FFFF$ on the compact set $\XX\times\OO$, we have
$$
\left\{\eta\in\OO:\dist(\eta,\NN(z))\le r\right\}\subset
\left\{\eta\in\OO:|\FFFF(z,\eta)|\le C'r\right\},
$$
where $C'>0$ does not depend on $z\in \XX$ and $r\in[0,1]$. The result follows from~\eqref{2.25}. 
\end{proof}

\begin{proof}[Proof of Lemma~\ref{l2.6}]
It is enough to prove~\eqref{2.25} for $r\le r_0$ with some $r_0>0$. Let $\ZZ(z,r)$ be the set in~\eqref{2.25}, and take any  $z^0\in \XX$ and $\eta^0\in \OO$. If~$\FFFF(z^0,\eta^0)\ne0$, then  
\begin{equation}\label{4.10}
|\FFFF(z,\eta)|> \sigma(z^0, \eta^0)>0, \quad  
\text{ $z\in O_{z^0}$, $\eta\in O_{\eta^0}$} 
\end{equation}for sufficiently small balls $\XX\supset O_{z^0}\ni z^0$ and~$E\supset O_{\eta^0}\ni\eta^0$.

Now, suppose that $\FFFF(z^0,\eta^0) =0$. Since $\eta^0$ is a regular point for~$\FFFF(z^0,\cdot)$, for some $j\in [1,M]$, the derivative of the function $t\mapsto f_{z^0,\eta^0}(t):=\FFFF(z^0,\eta^0+te_j)$ at $t=0$ is non-zero. Consequently,  there is $\delta>0$ and open balls $O_{z^0}\ni z^0$ and~$O_{\eta^0}\ni\eta^0$ such that 
$$
|f'_{z,\eta}(t)  |\ge \gamma(z^0, \eta^0)>0,\quad  
\text{ $z\in O_{z^0}$, $\eta\in O_{\eta^0}$, $|t|\le\delta$}. 
$$
This implies that 
\begin{equation}\label{4.11}
\Leb\bigl(\{t\in [-\delta,\delta]: |f_{z,\eta}(t)| \le r\}\bigr) 
\le C(z^0, \eta^0)\,r
\end{equation}
for $z\in O_{z^0}$, $\eta\in O_{\eta^0}$, and $r\in[0,1]$. 
Applying Fubini's theorem, we conclude from~\eqref{4.11} that
\begin{equation}\label{4.12}
\Leb\bigl(\ZZ(z,r) \cap O_{\eta^0}\bigr) \le C(z^0,\eta^0)\,r.
\end{equation}

Now, consider a finite covering of~$\XX \times \OO$ by sets of the form
 $O_{z^j} \times O_{\eta^j}$ for which either~\eqref{4.10} or~\eqref{4.12} holds with
  $z^0=z^j$ and $\eta^0=\eta^j$. Let~$r_0$ be the minimum of the constants~$\sigma(z^j,\eta^j)$, and~$C$ the maximum of $C(z^j, \eta^j)$.
  The estimate~\eqref{2.25} then follows from~\eqref{4.10} and~\eqref{4.12}. \end{proof}

\subsection{Proof of Theorem~\ref{T:3.1}}

 Theorem~\ref{T:3.1} extends Theorem~1.1 of~\cite{KNS20.0} in two
directions: the mapping $\eta\mapsto S(z,\eta)$ is no longer
required to be analytic, and the convergence rate is shown to be
exponential, thereby answering a question left open
in~\cite{KNS20.0}. A minor difference is that the phase
space~$\HH$ is now the product of a separable Hilbert space and a
compact manifold, rather than a Hilbert space alone; this does not
affect the proofs.

We denote by $\|\cdot\|$ the distance in~$\HH$. Under Hypothesis {\rm \hyperlink{H1}{(H$_1$)}}, there is a number $q_0\in (0, 1)$ such that
\begin{equation}\label{EE:7.1}
	\|S_1(z;\zeta)-S_1(z';\zeta)\|\le q_0^{-1}\|z-z'\|, \quad  z, z'\in \mathcal{X},\, \zeta\in \mathcal{K}.
\end{equation}
 Let us set 
$\mathbb{X}:=\mathcal{X}\times \mathcal{X}$. For any $q\in (0,q_0)$, $d>0$, and integer $n\ge0$, let us define the pairwise disjoint sets
$$
\mathbb{X}_{n}:=\left\{(z, z')\in \mathbb{X}: q^{n+1}d<\|z-z'\|\le q^nd\right\}.
$$We also set $$
\mathbb{X}_{\infty}:=\{(z,z): z\in \XX\};
$$ in expressions such as ``$\,\mathbb{X}_n$ for some $n\ge k+1$'', the index $n$ is allowed to take the value~$\infty$.
We further choose $d$ so small that the set 
$$
\mathbb{X}_{<0}:=\left\{(z,z')\in \mathbb{X}: d< \|z-z'\|\right\}
$$is non-empty, and we partition $\mathbb{X}_{<0}$ according to the following lemma.
\begin{lemma}\label{L:7.1}
There is a number $p\in (0,1)$ and 
pairwise disjoint Borel sets  $\mathbb{X}_{-N}, \ldots,\mathbb{X}_{-1}$ such that
\begin{itemize}
	\item[1)] $
\cup_{k=-N}^{-1}\mathbb{X}_{k}=\mathbb{X}_{<0},
$
\item[2)] for any $-N\le k\le -1$, any $(z,z')\in \mathbb{X}_{k}$, and independent random variables $\zeta$ and $\zeta'$ with law $\ell$, we have 
$$
\pP\{ (S(z,\zeta),S(z',\zeta'))\in \mathbb{X}_{m} \  \text{for some $m\ge k+1$}\}\ge p.
$$

\end{itemize}

\end{lemma}
\begin{proof}Let us define the set
$$
\mathbb{Y}_{-1}:=\{(z,z')\in \mathbb{X}: \exists \xi,\xi'\in \KK \ \text{s.t.}\ S(z,\xi),S(z',\xi')\in \mathring{B}_\HH(\hat z, d/4)\},
$$
which is non-empty in view of Hypothesis  \hyperlink{H2}{\rm(H$_2$)} applied with $\e=d/4$. For any~$k\le -1$, we then define the sets $\mathbb{Y}_{k}$ inductively by
$$ 
\mathbb{Y}_{k-1}:=\{(z,z')\in \mathbb{X}\setminus \cup_{j=k}^{-1}\mathbb{Y}_{j}: \exists \xi,\xi'\in \KK \ \text{s.t.}\ (S(z,\xi),S(z',\xi'))\in \mathbb{Y}_{k}\}.
$$
Hypothesis  \hyperlink{H2}{\rm(H$_2$)} 
implies that there is an integer $N:=N(d)\ge1$ (not exceeding the integer $m$ in \hyperlink{H2}{\rm(H$_2$)} corresponding to $\e=d/4$) such that $\mathbb{Y}_{k}\neq \emptyset$ for $-N\le k\le -1$ and
$$
\bigcup_{k=-N}^{-1}\mathbb{Y}_{k}\supset \mathbb{X}_{<0}.
$$
Indeed, for any $(z,z')\in\mathbb{X}$, Hypothesis \hyperlink{H2}{\rm(H$_2$)} with $\e=d/4$ provides controls $\zeta_1,\dots,\zeta_m\in\KK$ steering $z$ and $z'$ into $\mathring{B}_\HH(\hat z, d/4)$ in $m$ steps. Let $(z_i,z_i')$, $0\le i\le m$, be the corresponding pair trajectory. Then $(z_{m-1},z_{m-1}')\in\mathbb{Y}_{-1}$ by the definition of~$\mathbb{Y}_{-1}$, and, descending along the trajectory, each pair $(z_{i},z_{i}')$ with $i\le m-2$ belongs to $\cup_{k\le -1}\mathbb{Y}_{k}$: if $(z_{i+1},z_{i+1}')\in\mathbb{Y}_{k}$, then either $(z_i,z_i')$ already lies in some $\mathbb{Y}_{j}$, or it lies in $\mathbb{Y}_{k-1}$ by definition. In particular, $(z,z')=(z_0,z_0')\in \cup_{k\le -1}\mathbb{Y}_{k}$.

Using \eqref{EE:7.1}, we pick positive numbers $\e_{-N}, \ldots, \e_{-1}$ verifying  
\begin{itemize}
\item[$\bullet$] for any $(z,z')\in B_{\mathbb{X}}(\mathbb{Y}_{-1},\e_{-1})$, there are $\xi, \xi'\in\KK$ such that 
$$
\|S(z,\xi)-S(z',\xi')\|<d,
$$

\item[$\bullet$]  for any $(z,z')\in B_{\mathbb{X}}(\mathbb{Y}_{k},\e_{k})$ and $-N\le k\le -2$, there are $\xi, \xi'\in\KK$ such that 
$$
(S(z,\xi),S(z',\xi'))\in B_{\mathbb{X}}(\mathbb{Y}_{k+1},\e_{k+1}),
$$
\end{itemize}where $B_{\mathbb{X}}(\mathbb{Y},\e)$ is the closed $\e$-neighbourhood of a set $\mathbb{Y}\subset \mathbb{X}$. 
The fact that $\supp \ell=\KK$ and a simple compactness argument imply that there is a number $p\in (0,1)$ such that
\begin{itemize}
\item[$\bullet$]for any $(z,z')\in B_{\mathbb{X}}(\mathbb{Y}_{-1},\e_{-1})$,  
$$
\pP\{ \|S(z,\zeta)-S(z',\zeta')\|<d \}\ge p,
$$

\item[$\bullet$]  for any $(z,z')\in B_{\mathbb{X}}(\mathbb{Y}_{k},\e_{k})$ and $-N\le k\le -2$,  
$$
\pP\{(S(z,\zeta),S(z',\zeta'))\in B_{\mathbb{X}}(\mathbb{Y}_{k+1},\e_{k+1})\} \ge p,
$$
\end{itemize}where $\zeta$ and $\zeta'$ are independent random variables with law~$\ell$.
It is easy to see that the sets defined by
 $$
\mathbb{X}_{-1}:= B_{\mathbb{X}}(\mathbb{Y}_{-1},\e_{-1}) \cap \mathbb{X}_{<0},  
\quad \mathbb{X}_{k-1}:= (B_{\mathbb{X}}(\mathbb{Y}_{k-1}, \e_{k-1})\cap \mathbb{X}_{<0})\setminus \cup_{j=k}^{-1}\mathbb{X}_{j}
$$
 satisfy all the required properties.
 \end{proof}
 Next, we have the following result. 
 
\begin{proposition}\label{P:7.2}
Under the  assumptions of Theorem~\ref{T:3.1}, for any $\nu\in (0,1)$ and  $q\in (0,q_0)$, there are constants $d_0\in (0,1)$ and $C>0$ with the following property: for any $d\in (0, d_0)$, there is a number $p\in (0,1)$, a probability space $(\Omega, \mathcal{F}, \mathbb{P})$, and measurable functions~$v, v': \mathbb{X} \times \Omega\to \XX$ such~that the following assertions hold.
\begin{itemize}
  \item[{\rm (a)}] For any $(z, z')\in \mathbb{X}$, the laws of $v(z, z';\cdot)$ and $v'(z,z'; \cdot)$ coincide with ${P}_1(z, \cdot)$ and ${P}_1(z',\cdot)$. Moreover, $v(z, z; \cdot)=v'(z,z; \cdot)$ almost surely for any $z\in \mathcal{X}$.
   \item[{\rm (b)}] For any $(z, z')\in \mathbb{X}_n$, we have
      \begin{gather}
      \mathbb{P}\left\{(v(z,z'),v'(z,z'))\in \mathbb{X}_m \,\,\text{for some  $m\ge n+2$}\right\}\ge 1-\nu, \label{E:aa1aan} \\
      \mathbb{P}\left\{(v(z,z'),v'(z,z'))\in \mathbb{X}_m\,\,\text{for some  $m\le n-2$}\right\}\le C\|z-z'\|^{\frac12}, \label{E:aa2aan}
      \end{gather}where  $n\ge 0$ in \eqref{E:aa1aan} and $n\ge 1$ in \eqref{E:aa2aan}.
   \item[{\rm (c)}] For any $-N\le k\le -1$ and $(z,z')\in \mathbb{X}_{k}$, we have 
$$
\pP\{ (v(z,z'),v'(z,z'))\in \mathbb{X}_{m} \  \text{for some $m\ge k+1$}\}\ge p.
$$
\end{itemize}
\end{proposition} 
Assertions (a) and (b) are derived from Proposition~\ref{PP:7.4} using a coupling argument exactly as Theorem~2.2 is derived from Proposition~2.3 in~\cite{KNS20}; we omit the details. We note that, under Hypotheses
{\rm\hyperlink{H1}{(H$_1$)}}, {\rm\hyperlink{H3}{(H$_3$)}}, and
{\rm\hyperlink{H4}{(H$_4$)}} alone, the proof of
Proposition~\ref{PP:7.4} applies without any changes with ${\cal R}_1$
replaced by~$S_1$ and $\wh{\KK}^{z}$ by~$\KK^{z}$; it is this
version that is used here. The full version, involving the matrix
component, is applied in the proof of
Proposition~\ref{esitimate-triple-semigroup}.
Assertion (c) is immediate from Lemma~\ref{L:7.1} by taking
$$
v(z,z',\zeta):=S(z,\zeta), \quad v'(z,z',\zeta'):=S(z',\zeta'), \quad (z,z')\in \mathbb{X}_{<0},
$$
with independent random variables $\zeta$ and $\zeta'$ distributed as $\ell$.

Theorem~\ref{T:3.1} now follows from Proposition~\ref{P:7.2} via the
Kantorovich functional argument detailed in Section~2 of~\cite{KNS20}.

\subsection{Proof of Lemma~\ref{E:L3.3}}\label{SS:7.2.2}

Abusing notation, we denote by the same symbol $\|\cdot\|$ the distances in~$\HH$ and~$\HH\times   \M_d(\R)$. Under Hypotheses {\rm \hyperlink{H1}{(H$_1$)}} and {\rm \hyperlink{H5}{(H$_5$)}}, there is a number $q_0\in (0, 1)$ such that
$$\| {\cal R}_1(z;\zeta)- {\cal R}_1(z';\zeta)\|\le q_0^{-1}\|z-z'\|, \quad  z, z'\in \mathcal{X},\, \zeta\in \mathcal{K}.$$
 Let us set 
$\mathbb{X}:=\mathcal{X}\times \mathcal{X}$ and 
$\mathbb{Z}:=\wh{Z}\times \wh{Z}$, where $\wh{Z}:=\mathcal{X} \times \mathcal{A}(\KK\times \XX)$.
For any numbers $d>0$ and $q\in (0,q_0)$ and any integer $n\ge0$, let us introduce the pairwise disjoint sets
\begin{align*}
\mathbb{X}_{n}&:=\left\{(z, z')\in \mathbb{X}: q^{n+1}d<\|z-z'\|\le q^nd\right\},\\
\mathbb{Z}_{-1}&:=\left\{(y,y')\in \mathbb{Z}: d< \|y-y'\|\right\},\\
 \mathbb{Z}_n&:=\left\{(y,y')\in \mathbb{Z}: q^{n+1}d<\|y-y'\|\le q^nd\right\},\\
\mathbb{Z}_{\infty}&:=\{(y,y): y\in \wh{Z}\}. 
\end{align*} 
The proof of Lemma~\ref{E:L3.3} is based on the following result, whose proof is identical to that of assertions (a) and (b) in Proposition~\ref{P:7.2}.

\begin{proposition}\label{esitimate-triple-semigroup}
Under the  assumptions of Lemma~\ref{E:L3.3}, for any $\nu\in (0,1)$ and  $q\in (0,q_0)$, there are constants $d_0\in (0,1)$ and $C>0$ with the following property: for any $d\in (0, d_0)$, there is a probability space $(\Omega, \mathcal{F}, \mathbb{P})$ and measurable functions~$V, V': \mathbb{X} \times \Omega\to \wh{Z}$ such~that the following assertions~hold.
\begin{itemize}
  \item[{\rm (a)}] For any $(z, z')\in \mathbb{X}$, the laws of $V(z, z';\cdot)$ and $V'(z,z'; \cdot)$ coincide with $\wh{P}_1(z, \cdot)$ and $\wh{P}_1(z',\cdot)$. Moreover, $V(z, z; \cdot)=V'(z,z; \cdot)$ almost surely for any $z\in \mathcal{X}$.
   \item[{\rm (b)}] For any $(z, z')\in \mathbb{X}_n$, we have
      \begin{gather}
      \mathbb{P}\left\{(V(z,z'),V'(z,z'))\in \mathbb{Z}_m \,\,\text{for some  $m\ge n+2$}\right\}\ge 1-\nu, \label{E:aa1aa} \\
      \mathbb{P}\left\{(V(z,z'),V'(z,z'))\in \mathbb{Z}_m\,\,\text{for some  $m\le n-2$}\right\}\le C\|z-z'\|^{\frac12}, \label{E:aa2aa}
      \end{gather}where  $n\ge 0$ in \eqref{E:aa1aa} and $n\ge 1$ in \eqref{E:aa2aa}.
\end{itemize}
\end{proposition}

\begin{proof}[Proof of Lemma~\ref{E:L3.3}] Given $c>0$, we choose $q\in(0,q_0)$ and $\nu\in(0,1)$ so small that $q^{\frac12}+\nu q^{-1}\le c$, let $d_0$ be the corresponding constant in Proposition~\ref{esitimate-triple-semigroup}, and fix any $d\in (0,d_0)$. For $\|z-z'\|> qd$, the inequality~\eqref{E:ineq2} holds trivially with $C=2(qd)^{-\frac12}$, since the left-hand side is bounded by $2\|g\|_\infty\le 2(qd)^{-\frac12}\|g\|_{\infty}\|z-z'\|^{\frac12}$; the case $z=z'$ is also trivial. It therefore suffices to consider $0<\|z-z'\|\le qd$; then there is~$n\ge1$ such that $(z, z')\in \mathbb{X}_n$. Let us denote
\begin{align*}
\Omega_{z,z'}^{1,n}&:=\left\{\omega: (V(z,z'),V'(z,z'))\in \mathbb{Z}_m\,\,\text{for some  $m\ge n+2$}\right\}, \\
\Omega_{z,z'}^{2,n}&:=\left\{\omega: (V(z,z'),V'(z,z'))\in \mathbb{Z}_m\,\,\text{for some  $m\le n-2$}\right\},\\
\Omega_{z,z'}^{3,n}&:=\Omega\setminus (\Omega_{z,z'}^{1,n}\cup \Omega_{z,z'}^{2,n})
\end{align*}and 
for any $g\in C_b^{\frac12}(\XX \times  \Sl_d(\R))$, decompose
\begin{align*}
\int g(y)\wh{P}_1(z, \dd y)-\int g(y)\wh{P}_1(z', \dd y)
& =\mathbb{E}\left(\mathbb{I}_{\Omega_{z,z'}^{1,n}}\left(g(V(z, z'))-g(V'(z, z'))\right)\right)\\ \notag
&\quad+\mathbb{E}\left(\mathbb{I}_{\Omega_{z,z'}^{2,n}}\left(g(V(z, z'))-g(V'(z, z'))\right)\right)\\ \notag
&\quad+\mathbb{E}\left(\mathbb{I}_{\Omega_{z,z'}^{3,n}}\left(g(V(z, z'))-g(V'(z, z'))\right)\right)\\ \notag
& =:  I_1+I_2+I_3.
\end{align*}
We estimate $I_1$ as follows:
\begin{align*}
|I_1|&\le |g|_{\frac12}\mathbb{E}\left(\mathbb{I}_{\Omega_{z,z'}^{1,n}} \|V(z, z')-V'(z,z')\|^{\frac12}\right) 
\le  |g|_{\frac12}\left(q^{n+2}d\right)^{\frac12} \\
&\le |g|_{\frac12} q^{\frac12}\left(q^{n+1}d\right)^{\frac12} 
<  |g|_{\frac12}q^{\frac12}\|z- z'\|^{\frac12},
\end{align*}
where the second inequality uses that $$\|V(z,z')-V'(z,z')\|\le q^{n+2}d$$ on $\Omega_{z,z'}^{1,n}$, and the last inequality follows from the fact that $(z,z')\in \mathbb{X}_n$, which gives $q^{n+1}d<\|z-z'\|\le q^nd$.
To estimate $I_2$, we use assertion (b) in Proposition \ref{esitimate-triple-semigroup}: 
$$
|I_2|\le2\|g\|_{\infty}\mathbb{P}(\Omega_{z,z'}^{2,n})\le2\|g\|_{\infty}C\|z-z'\|^{\frac12}.
$$
Finally, we note that on the event $\Omega_{z,z'}^{3,n}$ there holds
$$(V(z, z'), V'(z, z'))\in \mathbb{Z}_{n-1}\cup \mathbb{Z}_n\cup\mathbb{Z}_{n+1}\cup\, \mathbb{Z}_{\infty},$$
therefore,
\begin{align*}
\|V(z, z')- V'(z, z')\|\le q^{n-1}d=q^{-2}(q^{n+1}d)\le q^{-2}\|z- z'\|
\end{align*}
and
$\mathbb{P}(\Omega_{z,z'}^{3,n})\le \nu.$
It follows that
$$
|I_3|\le  |g|_{\frac12}\mathbb{P}(\Omega_{z,z'}^{3,n})q^{-1}\|z- z'\|^{\frac12}
\le   |g|_{\frac12} \nu q^{-1} \|z- z'\|^{\frac12}.
$$
Thus, combining the above estimates for $I_1, I_2,$ and $I_3$, we arrive at
$$
|I_1|+|I_2|+|I_3|\le \left(q^{\frac12}+\nu q^{-1}\right) |g|_{\frac12} \|z- z'\|^{\frac12}+ 2C\|g\|_{\infty} \|z- z'\|^{\frac12}.
$$
By the choice of $q$ and $\nu$, this gives \eqref{E:ineq2} with $C(c):=\max\{2C, 2(qd)^{-\frac12}\}$, completing the proof of the lemma.
\end{proof}

\section{Application to the Navier--Stokes system}\label{S:4}

Here, we present a more general version of the Main Theorem formulated in the Introduction and prove it by using the abstract result of the previous section, along with the controllability properties obtained in Sections~\ref{S:5} and~\ref{S:6}.

\subsection{Formulation}\label{S:4.1}

 Let us consider the system~\eqref{E:1.1},~\eqref{E:1.2} driven by a random process $\eta$ that has independent and stationary increments and takes values in a finite-dimensional space $\eE$. We assume that the space~$\eE$ has the specific form
$$
\eE:= \textrm{span} \{E_j: j\in \mathcal{I}\},
$$ where $E_j$ are given by
\begin{equation}\label{EEE:f}
E_j(x):=j^{\bot}
\begin{cases}
\cos\langle j, x  \rangle \ \textrm{for} \ j_1>0 \ \textrm{or} \ j_1=0, \ j_2>0, \\
\sin\langle j, x \rangle \ \textrm{for} \ j_1<0 \ \textrm{or} \ j_1=0, \ j_2<0
\end{cases}
\end{equation}
with $j^{\bot}=(-j_2, j_1)$, and the index set $\mathcal{I}\subset \mathbb{Z}^2_*$ is defined by
 \begin{equation}\label{E:4.4}
\mathcal{I}:=\left\{j=(j_1, j_2)\in \mathbb{Z}^2_*: |j_i|\leq 1,\, i=1, 2\right\}.
\end{equation}
 The process $\eta$ is assumed to be of the~form  
\begin{equation}\label{E:4.1}
\eta(t, x)=\sum_{k=1}^{\infty} \mathbb{I}_k(t) \eta_k(t-k+1, x),
\end{equation}
where $\mathbb{I}_k$ is the indicator function of the interval $[k-1, k)$ and $\{\eta_k\}$ is a sequence of i.i.d. random variables in $L^2(J, \eE)$ with $J:=[0,1]$ defined on a probability space $(\Omega_0, \mathcal{F}_0, \bbp_0)$. We denote by~$\ell \in \PP(L^2(J, \eE))$ the common law of the random variables $\{\eta_k\}$ and assume that $\KK:=\supp \ell$ is compact in~$L^2(J, \eE)$. Recall that $V^5=H^5\cap H$ is the space of divergence-free
vector fields on $\bbt^2$ of Sobolev class $H^5$ with zero mean
value, introduced in Section~\ref{Sec-Itro-Main}.

Under these assumptions, for any initial data $(u_0, x)\in V^5\times \bbt^2$, the system~\eqref{E:1.1}, \eqref{E:1.2} has a unique solution~$(u(t), \phi^t)$. The incompressibility condition $\div\, u=0$ implies that the Jacobian $D_x\phi^t$ of $\phi^t$ at any point $x\in \bbt^2$ belongs to~$\Sl_2(\R)$.

The {\it Lagrangian process}~$(u_n, y_n)$ obtained by restricting $(u(t), \phi^t)$ to integer times is Markovian. In what follows, we will study this process within the framework of the previous section by setting~$\HH:= V^5 \times \bbt^2$ and~$E:=L^2(J, \eE)$, and considering the time-1 shift along the trajectories of the system \eqref{E:1.1}, \eqref{E:1.2}:
\begin{equation}\label{E:4.2}
S: \HH\times E\to \HH, \ \  ((u,x), \eta)\mapsto (u_1, y_1)	
\end{equation}
  and the {\it derivative cocycle} generated by the mapping
$$
\aA: \HH\times E\to \Sl_2(\R), \quad  ((u,x), \eta)\mapsto \aA_{\eta, u,x}:=D_x\phi^1.
$$
  We choose~$V^5$ as the phase space for the velocity component since then $S$ and $\aA$ are well defined and twice continuously differentiable from $\HH\times E$ to~$V^6\times \bbt^2$ and from  $\mathcal{H}\times E$ to~$\Sl_2(\R)$, respectively, with bounded derivatives on bounded subsets (cf. Hypotheses~\hyperlink{H1}{\rm(H$_1$)} and {\rm \hyperlink{H5}{(H$_5$)}}).  We have the relations
\begin{equation}\label{E:4.3} 
(u_n, y_n) =S((u_{n-1}, y_{n-1}), \eta_n), \ \ n\ge 1.
\end{equation}
Let us define a sequence of sets as follows:
$$
 X_0:=\{0\}, \ X_n:=S^u(X_{n-1}, \mathcal{K}),\ n\geq 1
$$with $S^u$ being the velocity component of $S$, and denote by $X$ the closure\footnote{Given our assumptions on the structure of the space $\eE$, the union $\cup_{n\geq0}{X_n}$ belongs to~$V^m$ for any $m\ge1$. 
By defining $X$ as the $V^8$-closure of this union, we simplify the argument of the proof of the approximate controllability of the linearized system in Section~\ref{S:6}. Taking~$V^5$ would also work, provided that an additional result is used on the controllability of linear PDEs with initial data as a control, in the spirit of Proposition 7.2 in~\cite{KNS20}.} in~$V^8$ of the union $\cup_{n\geq0}{X_n}$. By the dissipativity of the Navier--Stokes
dynamics, the boundedness of $\mathcal K$ in~$E$, and the
parabolic regularization, the union $\bigcup_{n\geq1}X_n
=S^u\bigl(\bigcup_{n\geq0}X_n,\KK\bigr)$ is bounded in~$V^9$;
hence so is its $V^8$-closure~$X$. By interpolation, $X$ is
closed in~$V^5$, and thus compact in~$V^5$. It follows that
$\XX:= X\times \bbt^2$ is compact in~$\HH$. Furthermore, by the
definition of~$X$ and the continuity of $S^u$ from $V^8\times E$
to~$V^8$, the set $\XX$ is invariant under the dynamics:
$$
S(\XX\times \mathcal{K})\subset \XX.
$$
In what follows, we consider the restriction of the
RDS~\eqref{E:4.3} to~$\XX$.

To be able to apply Theorem~\ref{T:3.2}, we need additional assumptions on the law $\ell$.~To this end, let us recall the notion of {\it observable functions} from~\cite{KNS20}.

 \begin{definition}\label{D:4.1}
 Let the finite-dimensional space $\eE$ be endowed with an inner product $(\cdot, \cdot)_{\eE}$ and an orthonormal basis $\{e_j\}_{j\in \mathcal{I}}$. A function $\zeta\in L^2(J,\eE)$ is said to be {\it Lipschitz-observable} if for any Lipschitz-continuous functions $a_j: J\rightarrow \R$, $j\in \mathcal{I}$ and any continuous function $b: J\rightarrow \R$  the equality
 $$
 \sum_{j\in \mathcal{I}}a_j(t)(\zeta(t), e_j)_{\eE}+b(t)=0\ \textrm{in} \ L^2(J,\R)
 $$
implies that $a_j\equiv b\equiv 0$ on $J$ for $j\in \mathcal{I}$. 
\end{definition}
It is easy to see that the observability property does not depend on the choice of the basis $\{e_j\}_{j\in \mathcal{I}}$ (see Remark~4.2 in~\cite{KNS20}). In what~follows, we~denote by $\{e_j\}_{j\in \mathbb{Z}^2_*}$ the orthonormal basis in $V^5$ obtained by normalizing the functions $E_j$ in \eqref{EEE:f}, and we endow the space $\eE$ with the inner product induced from~$V^5$, so that $\{e_j\}_{j\in \mathcal{I}}$ is an orthonormal basis of~$\eE$. The following hypothesis gathers all the requirements on the sequence   $\{\eta_k\}$.

\begin{description} 
\item[\hypertarget{N}{(N)}] 
{\sl The random variables $\{\eta_k\}$ are of the form
$$
\eta_k(t,x)=\sum_{j\in \mathcal{I}} \sum_{l=1}^\ty b_l \xi_{jl}^k \psi_l(t)e_j(x),
$$where 
\begin{itemize}
	\item $\mathcal{I} \subset \mathbb{Z}^2_*$ is the set defined by \eqref{E:4.4}, 
	\item  $\{b_l\}$ are non-zero numbers such that 
$\sum_{l=1}^\ty b_l^2<\ty$, 
	\item  $\{ \xi_{jl}^k\}$ are independent scalar random variables such that $| \xi_{jl}^k| \le  1$ a.s. and 
$$
\DD(\xi_{jl}^k)=\rho_{jl}(r)\dd r
$$ with Lipschitz-continuous density $\rho_{jl}$ such that $\rho_{jl}(0)>0$, 
	\item   $\{\psi_l\}$ is an  orthonormal basis in $L^2(J,\R)$ such that~$\psi_1(t)=1$,~$t\in J$.
\end{itemize}
 Moreover, $\{\eta_k\}$ are observable a.s. in the sense of Definition~\ref{D:4.1}.} 
\end{description} 
In particular, this hypothesis guarantees that the support $\KK$ of the law $\ell$ of the random variables $\{\eta_k\}$ is compact in $E$ and $0\in\KK$.

The following is a more general version of the Main Theorem formulated in the Introduction.
\begin{theorem}\label{T:4.2}
Under Hypothesis {\rm \hyperlink{N}{(N)}}, the Lagrangian process $(u_n, y_n)$ has a unique stationary measure, given by $\mu\times \leb \in \PP(\mathcal{X})$, which is exponentially mixing and whose support is equal to $\XX$. Moreover, there exists a deterministic constant $\lambda_+>0$ such that the following limit holds:
$$
\lambda_+=\lim_{n\rightarrow \infty}\frac 1n\log|D_x\phi^n|
$$for $\mu\times \leb\times \pP_0$-a.e. $(u,x,\omega)\in X\times \bbt^2\times \Omega_0$.
\end{theorem}
The Haar coloured noise mentioned in the Introduction satisfies
Hypothesis~{\rm \hyperlink{N}{(N)}}. All the properties except the
observability follow directly from the construction; the
observability is established in \cite[Section~5.2]{KNS20}, and here we only sketch the idea. The processes $\eta^j$
have independent random jumps at the dyadic points of every level,
with amplitudes that have densities. If a relation
$\sum_{j\in\mathcal{I}}a_j(t)\eta^j(t)+b(t)=0$ held with
Lipschitz-continuous $a_j$ and continuous~$b$, then, since $a_j$ and~$b$ are continuous and cannot produce jumps, the coefficients $a_j$ would have to cancel the jumps of the noise at every dyadic point. As the jumps at different points and in different components are independent and non-degenerate, each such cancellation is an event of probability zero. A quantitative version of this argument, based on small-ball estimates for the jump amplitudes, shows that almost surely $a_j\equiv0$ for $j\in\mathcal{I}$, and then $b\equiv0$. Consequently, the Main
Theorem is a particular case of Theorem~\ref{T:4.2}.

The ergodicity of the Lagrangian process in the case of a non-degenerate bounded noise has been shown in~\cite{JNPS21}; in that setting, no observability assumption on the noise is needed.

\subsection{Proof of Theorem~\ref{T:4.2}}

 Theorem~\ref{T:4.2} is established by applying Theorem~\ref{T:3.2} within the framework described in the previous subsection. Thus, it is necessary to check the validity of Hypotheses \hyperlink{H1}{\rm(H$_1$)}--\hyperlink{H7}{\rm(H$_7$)}. We take $\hH=V^5$ and $\vV=V^6$, so that~$\vV$ is compactly
embedded into~$\hH$, as required in \hyperlink{H1}{\rm(H$_1$)}. The remaining regularity properties
in~\hyperlink{H1}{\rm(H$_1$)} and {\rm \hyperlink{H5}{(H$_5$)}}
follow from the parabolic regularization of the Navier--Stokes
system and the smooth dependence of the solution on the
right-hand side and the initial data (e.g., see~\cite{BV92}). The approximate controllability of the nonlinear system to a point~\hyperlink{H2}{(H$_2$)} is verified in Corollary~\ref{C:5.6}, the decomposability property~\hyperlink{H4}{(H$_4$)} follows immediately from Hypothesis~\hyperlink{N}{(N)}, and the approximate controllability of the linearized system~\hyperlink{H7}{(H$_7$)} is checked in Theorem~\ref{T:6.1}.~As already mentioned,~\hyperlink{H7}{(H$_7$)} implies~\hyperlink{H3}{(H$_3$)}.
 Therefore, the conditions of Theorem~\ref{T:3.1} are satisfied, so~$\mu\times \leb$ is an exponentially mixing stationary measure for the Lagrangian process~$(u_n,y_n)$.~Corollary~\ref{C:5.6} implies that~$\supp \,(\mu\times \leb)=\XX$. Finally, the approximate controllability of the extended system {\rm \hyperlink{H6}{(H$_6$)}} is verified in 
Theorem~\ref{T:5.1}. Thus, applying Theorem~\ref{T:3.2}, we complete the proof of Theorem~\ref{T:4.2}.

 \subsection{Projective process}\label{S:4.3}

Let us consider the projective equation 
\begin{equation}\label{EE:vt}
	  \dot{v}_t   = \Pi_{v_t} D_xu(t,y(t)) v_t,  
\end{equation}where $\Pi_{v}$ is the orthogonal projection onto ${\mathsf T}_{v} P^1$ (the tangent to the projective space $P^1$ at $v$).~Using the ergodicity of the projective process $(u_n,y_n,v_n)$, one can derive the following improvement of Theorem~\ref{T:4.2} in the spirit of~Corollary~1.7 in~\cite{BBS22}.
   \begin{theorem}\label{T:4.3T}
Under the conditions of Theorem~\ref{T:4.2}, the projective process $(u_n, y_n,v_n)$ has a unique stationary measure, which is exponentially mixing. Moreover, for any non-zero vector $v\in \R^2$, the following limit holds:
\begin{equation}\label{E:projl}
\lambda_+=\lim_{n\rightarrow \infty}\frac 1n\log|D_x\phi^nv|	
\end{equation}
for $\mu\times \leb\times \pP_0$-a.e. $(u,x,\omega)\in \XX\times \Omega_0$. 
\end{theorem}
The exponential mixing for the projective process $(u_n,y_n,v_n)$ is established using Theorem~\ref{T:3.1} in a manner similar to that of the Lagrangian process $(u_n,y_n)$. Here, the phase space is taken as~$\HH= V^5 \times \mM$, where $\mM=\bbt^2\times P^1$, and the verification of Hypotheses \hyperlink{H1}{\rm(H$_1$)}--\hyperlink{H3}{\rm(H$_3$)} proceeds along the same lines. The verification of the controllability hypotheses reduces to the results of Sections~\ref{S:5} and~\ref{S:6} as follows. For a fixed direction $v_0\in P^1$, let us define the map
$$
\pi_{v_0}:\Sl_2(\R)\to P^1, \qquad \pi_{v_0}(A):=Av_0,
$$
where $A$ also denotes the map induced on~$P^1$ by the matrix~$A$ (see the Notation). If $(u(t),y(t),A(t))$ is a trajectory of the system \eqref{E:5.1}--\eqref{E:5.3}, then~$v(t):=\pi_{v_0}(A(t))$ solves the projective equation \eqref{EE:vt} with $v(0)=v_0$; that is, the mapping $(u,y,A)\mapsto(u,y,\pi_{v_0}(A))$ sends the trajectories of the matrix process to those of the projective process. The map $\pi_{v_0}$ is smooth and surjective (since the rotation group $\mathrm{SO}(2)\subset\Sl_2(\R)$ acts transitively on~$P^1$), and its differential $D\pi_{v_0}(A):{\mathsf T}_A\Sl_2(\R)\to {\mathsf T}_{Av_0}P^1$ is surjective for every~$A$. Consequently, the approximate controllability of the projective triple $(u,y,v)$ follows from Theorem~\ref{T:5.1}: it suffices to steer the matrix component close to a rotation taking $v_0$ to the target direction. Moreover, the control time can be chosen independently of the data. Indeed, all the matrices involved lie in a compact subset of $\Sl_2(\R)$, so the number of steps required in Propositions~\ref{P:5.2} and~\ref{P:5.3} is uniformly bounded, and any trajectory can be prolonged by zero controls, under which the triple stays at rest once the velocity vanishes. The density of the image of the linearization follows from Theorem~\ref{T:6.1}, since the linearized endpoint map of the projective system is obtained by composing that of the matrix system with the surjective linear map $\Id\times\Id\times D\pi_{v_0}(A(1))$. Thus, all the hypotheses of Theorem~\ref{T:3.1} are satisfied, and the exponential mixing of the projective process follows.

The limit \eqref{E:projl} follows from the ergodicity of the projective process combined with a version of the multiplicative ergodic theorem given in Theorem~III.1.2 in~\cite{Kif86}.

\section{Controllability of the nonlinear Lagrangian flow}\label{S:5}

In what follows, we use the notation of the previous sections and
 assume that the conditions of Theorem~\ref{T:4.2} are satisfied. 
The~aim of this section is to check the validity of Hypotheses \hyperlink{H2}{(H$_2$)} and~{\rm \hyperlink{H6}{(H$_6$)}} for the Lagrangian process. To this end, we consider the controllability of the nonlinear system  with Navier--Stokes equations written in the projected~form:
\begin{gather} 
  \partial_t u + \LLLL u + B(u)= \eta(t,x),   \label{E:5.1} \\
  \dot{y}(t)  = u(t,y(t)), \label{E:5.2}\\
  \dot{A}(t)  = D_xu(t,y(t)) A(t), \label{E:5.3}
\end{gather}supplemented with the initial condition
\begin{equation}\label{E:5.4}
u(0) =u_0, \quad 	  y(0) =y_0, \quad A(0) = \Id_{\R^2}.
\end{equation}Here $\LLLL:= -\nu \Delta$ is the Stokes operator, $B(u):= \Pi \left( \<u, \na\> u\right)$ is the nonlinear term, and  $\Pi:L^2(\bbt^2,\R^2)\to H$ is the Leray projection. For any $(u_0, y_0)\in \HH$ and $\eta\in L^2([0,T],\eE)$, the system \eqref{E:5.1}-\eqref{E:5.4}  has a unique solution $( u, y, A)$ in $C([0,T],\wh{\HH})$, where $\wh{\HH}:=\HH\times \Sl_2(\R)$. Theorem \ref{T:5.1} implies that Hypothesis~{\rm \hyperlink{H6}{(H$_6$)}} is verified. 
\begin{theorem} \label{T:5.1}
For any $\ve>0$, any initial data $(u_0,y_0) \in \XX$, and any target 
$(u^\sharp, y^\sharp, A^\sharp) \in \XX \times\Sl_2(\R)$,
there is $m\ge 1$ and controls $\zeta_1, \ldots, \zeta_m\in \mathcal{K}$ such that
\begin{equation}   
d_{\wh{\HH}}\left({\cal R}_m((u_0, y_0, \Id_{\R^2}); \zeta_1, \ldots, \zeta_m) , (u^\sharp, y^\sharp, A^\sharp)\right) < \ve.\label{E:5.5}
\end{equation}
\end{theorem}

As a preparation for the proof of this theorem, we establish two
propositions. The first one shows that it is possible to control
exactly the particle position~$y(t)$, while keeping the velocity
$u(t)$ and the matrix $A(t)$ in the same place.
The construction of the controls is based on
shear flows and is inspired by~\cite[Lemma~7.1]{BBPS22c}. The main
new difficulty is the admissibility of the controls: for an
arbitrary shear profile~$f$, the corresponding time coefficient
$g:=f'+\nu f$ does not necessarily generate a force belonging to the compact
set~$\mathcal K$. Furthermore, in order to concatenate the
controls over successive unit time intervals, the profile must
satisfy $f(0)=f(1)=0$, while producing the prescribed
displacement $\int_0^1f(s)\,\dd s=a_1$. These constraints are
met simultaneously by choosing the force of the form
$$
g(t)=\nu a_1+a_j\psi_j(t),
$$
where $j\geq2$ is such that $\int_0^1e^{\nu s}\psi_j(s)\,\dd s
\neq0$, and the coefficient~$a_j$ is chosen so that the solution
of the equation $f'+\nu f=g$ with $f(0)=0$ satisfies also
$f(1)=0$; as $\int_0^1\psi_j(s)\,\dd s=0$, this automatically
yields $\int_0^1f(s)\,\dd s=a_1$. Finally, the coefficients
$a_1$ and $a_j$ are proportional to the prescribed displacement,
so for small displacements the resulting control belongs
to~$\mathcal K$, due to the positivity of the densities
$\rho_{jl}$ at the origin.
 
\begin{proposition}\label{P:5.2} 
There is an integer $m\ge1$ with the following property: for any $y_0, y^\sharp \in \bbt^2$ and $A_0\in \Sl_2(\R)$, there are controls $ \zeta_1,\ldots,  \zeta_{m}\in \mathcal{K}$ such~that
$$
{\cal R}_{m}((0,y_0, A_0);\zeta_1, \ldots, \zeta_{m})=(0, y^\sharp, A_0).
$$
\end{proposition}
\begin{proof}
Using a compactness argument and the fact that $u\equiv 0$ is a solution of the Navier--Stokes system~\eqref{E:5.1}, we see that it suffices to prove the following local version of the result: there is a number $\kappa>0$ such that, for any~$y_0,  y^\sharp\in \bbt^2$ with $|y_0-y^\sharp|<\kappa$, any $ A_0\in \Sl_2(\R)$,  and appropriate controls~$ \zeta_1,  \zeta_2\in \mathcal{K}$, we have
$$
S_2((0,y_0);  \zeta_1, \zeta_2)= (0, y^\sharp) \quad \text{and}\quad    A(t) \equiv  A_0,\ \   t\in [0, 2].
$$   

 First, let us show that we can shift the $y$-component horizontally; that is, 
  let us take sufficiently close points $y_0, \hat{y}\in \bbt^2$ of the form $y_0=(y_{01}, y_{02})$ and $\hat{y}=(y_1^\sharp, y_{02})$ and construct $\zeta_1\in \mathcal{K}$ such that $S((0,y_0);\zeta_1)=(0, \hat{y})$. Indeed, since~$\{\psi_l\}_{l\geq1}$ is an orthonormal basis in $L^2(J,\R)$ with $\psi_1=1$  (see Hypothesis~\hyperlink{N}{(N)}), we have that $\int_0^1\psi_l(s)\dd s=0$ for any $l\geq2$, and
there is~$j\geq2$ such that~$\int_0^1e^{\nu s}\psi_j(s)\dd s\neq0$, where~$\nu>0$ is the viscosity in \eqref{E:5.1}. Then, the function~$g$ defined by 
\begin{equation}\label{E:5.6}
g(t):=\nu a_1 +a_j\psi_j(t), \quad t\in J 
\end{equation}
 with coefficients $a_1$ and $a_j$ given by 
 \begin{equation}\label{E:5.7}
 a_1:= y_1^\sharp-y_{01} \ \ \textrm{and} \ \ a_j:=\frac{a_1(1-e^{\nu})}{\int_0^1e^{\nu s}\psi_j(s)\dd s}
 \end{equation}
satisfies the relations
$$
\int_0^1g(s)\dd s=\nu a_1, \ \
\int_0^1e^{\nu s}g(s)\dd s=0.
$$
It follows that the function 
\begin{equation}\label{E:5.8}
f(t):= \int_0^te^{-\nu(t-s)}g(s)\dd s, \quad t\in J
\end{equation}
satisfies
\begin{gather*}
	f(0)=f(1)=0, \ \ \int_0^1f(s)\dd s=a_1,\\
f'(t)+\nu f(t)=g(t), \quad t\in J.
\end{gather*}
Following the shear-flow construction in \cite[Lemma~7.1]{BBPS22c},
let us define
$$
   u(t,x) := f(t)\begin{pmatrix} \cos (x_2-y_{02}) \\ 0\end{pmatrix}, \quad (t,x)\in J\times \bbt^2.
$$
This flow has the following properties:
\begin{itemize}
	\item[(a)]  $u(0)= u(1)=0$,
	\item[(b)]  the solution $y(t)=(y_1(t),y_2(t))$ of
	\begin{equation*}
\left\{ \begin{aligned} 
  &  \dot{y}(t) = u(t, y(t)),    \\
    & y(0) =  y_0
\end{aligned}\right.
\end{equation*}
satisfies $y_1(1) =  y_1^\sharp$  and $y_2(t) = y_{02}$, $t\in [0,1]$,
	\item[(c)] Since $u(t,x)$ defined above is a one-directional shear flow, the nonlinear term $(u\cdot \nabla)u$ of the Navier--Stokes system \eqref{E:5.1} vanishes. Therefore, the corresponding control $\zeta_1$ is given explicitly by
 \begin{align*}
\zeta_1(t,x)&=(f'(t)+\nu f(t))\cos y_{02} 
                \begin{pmatrix}
                       \cos x_2 \\
                       0 
                     \end{pmatrix}
\\&\quad+ (f'(t)+\nu f(t))\sin y_{02}
                     \begin{pmatrix}
                       \sin x_2 \\
                       0 \\
                     \end{pmatrix}
                  \\
&=g(t)\cos y_{02}                  \begin{pmatrix}
 \cos x_2 \\
 0 \\
\end{pmatrix}
+ g(t)\sin y_{02}
            \begin{pmatrix}
             \sin x_2 \\
             0 \\
             \end{pmatrix}                   .
\end{align*}
 \end{itemize}
The additional point here is to verify that the resulting control $\zeta_1$ is admissible for the bounded noise, namely, that $\zeta_1\in\mathcal K$. Indeed, by choosing $\kappa>0$ sufficiently small, the coefficients $a_1$ and $a_j$ in~\eqref{E:5.7} can be made small enough; the assumption $\rho_{jl}(0)>0$ (see Hypothesis {\rm \hyperlink{N}{(N)}}) then ensures that $\zeta_1$ belongs to $\mathcal K$.

  To see that the control $\zeta_1$ does not affect the matrix flow, we note that the above-defined shear flow $u$ satisfies
$$
   D_xu(t,x)
= f(t) 
     \begin{pmatrix}
       0 &  -\sin(x_2-y_{02}) \\
       0 & 0 \\
     \end{pmatrix}.
$$
Evaluating along the Lagrangian flow $y(t)$ appearing in property (b), we~derive
\begin{align*}
  \dot{A}(t) =  D_xu(t,y(t))A(t) = 0,
\end{align*}
which yields that $A(t) = A_0$ for   $t\in [0, 1]$.

In a similar way, we can shift the $y$-component vertically; that is,
 for sufficiently close points $\hat{y},  y^\sharp\in \bbt^2$ of the form $\hat{y}:=(y_1^\sharp, y_{02})$ and $ y^\sharp:=(y_1^\sharp,y_2^\sharp)$, we choose $g$ as in \eqref{E:5.6} with $a_1:=y_2^\sharp-y_{02}$ in~\eqref{E:5.7}. Then, we define $f$ as in~\eqref{E:5.8} and take the shear flow
$$
   u(t,x) = f(t)  
                     \begin{pmatrix}
                      0 \\
                       \cos (x_1-  y_1^\sharp)  \\
                     \end{pmatrix}
                   .
$$
As above, the following properties are satisfied: 
\begin{itemize}
	\item[(a$^\prime$)]  $u(0)= u(1)=0$;
	\item[(b$^\prime$)]  the solution $y(t)=(y_1(t),y_2(t))$ of 
	\begin{equation*}
\left\{ \begin{aligned} 
  &  \dot{y}(t) = u(t, y(t)),    \\
    & y(0) =  \hat{y}
\end{aligned}\right.
\end{equation*}
verifies $y_1(t)=  y_1^\sharp$, $t\in J$ and $y_2(1) =  y_2^\sharp$;
	\item[(c$^\prime$)] $u(t,x)$ is a solution of \eqref{E:5.1} with control given by 
	$$
   \zeta_2(t,x) = g(t)\cos  y_1^\sharp\begin{pmatrix}
                     0 \\
                       \cos x_1 \\
                     \end{pmatrix}
               +g(t)\sin  y_1^\sharp   \begin{pmatrix}
                      0 \\
                       \sin x_1\\
                     \end{pmatrix}
                   .
$$
 \end{itemize}
 For $\kappa$ small enough, we have $\zeta_2\in \mathcal{K}$, and the corresponding matrix flow satisfies $A(t) =  A_0$ for   $t\in [0, 1]$. Hence, combining the horizontal and vertical shifts, we conclude that, for any~$y_0, y^\sharp\in \bbt^2$ with $|y_0- y^\sharp|<\kappa \ll 1$, we have $S_2((0, y_0); \zeta_1, \zeta_2)=(0, y^\sharp)$.  This completes the proof.
\end{proof}
 The second proposition shows that it is possible to control exactly the matrix~$A(t)$,
 while keeping the velocity $u(t)$ and the particle position~$y(t)$ 
   in the same place.     
\begin{proposition}\label{P:5.3}
For any matrices $ A_0,  A^\sharp\in\Sl_2(\R)$,
there is an integer $m\ge1$, depending on  $A_0$ and $A^\sharp$, and controls  $\zeta_1,\ldots,\zeta_{m}\in \mathcal{K}$
such that
\begin{equation}
 \label{E:5.9}
    {\cal R}_{m}((0, \wt y, A_0); \zeta_1,\ldots,\zeta_{m}) = (0,\wt y, A^\sharp),
\end{equation}where $\wt y:= (\pi/2, \pi/2)\in\bbt^2$.
\end{proposition}
The proof of this proposition follows from the two lemmas below. Let us recall that the transvection or shear matrices are of the form
$$
  T_{12}(\alpha)
  = 
     \begin{pmatrix}
       1 & \alpha \\
       0 & 1 \\
     \end{pmatrix}
   , \ \
  T_{21}(\beta)=
     \begin{pmatrix}
       1 & 0 \\
       \beta & 1 \\
     \end{pmatrix}
, \ \
   \alpha, \beta\in \bbr,
$$and let $\Tv_2(\R)$ be the collection of all such matrices:
$$
\Tv_2(\R):=\left\{T_{12}(\alpha), \ T_{21}(\beta),\,\,\alpha, \beta\in 
\bbr\right\}.
$$  
The following lemma is straightforward to verify.
\begin{lemma}\label{L:5.4}
Transvection matrices generate the special linear group $\Sl_2(\R)$. More precisely, for any $A\in \Sl_2(\R)$, there are matrices $T_1,\ldots, T_{4}\in \Tv_2(\R)$ such that $A=T_1\circ\cdots\circ T_{4}$.
\end{lemma}

As in Proposition~\ref{P:5.2}, additional care is needed to ensure
that the controls belong to the compact set $\mathcal K$. To this
end, we write an arbitrary transvection as a finite product of
transvections with small parameters, each of which is realised by
an admissible shear control. 
\begin{lemma}\label{L:5.5}
For any $\alpha\in \bbr$ and $ A_0\in\Sl_2(\R)$, there is an integer $m_{\alpha}\ge1$ depending on $\alpha$, and controls $\zeta_1,\ldots, \zeta_{m_{\alpha}}\in \mathcal{K}$ such that
\begin{equation}\label{E:5.10}
\begin{aligned}  
  {\cal R}_{m_{\alpha}}((0, \wt y, A_0); \zeta_1,\ldots, \zeta_{m_{\alpha}})
     &= (0,\wt y, T_{12}(\alpha) A_0),
\end{aligned}
\end{equation}
where $\wt y\in\bbt^2$ is as in Proposition~\ref{P:5.3}.
Similarly, for any $\beta\in \R$, there is an integer $m_{\beta}\ge1$ depending on $\beta$, and controls 
 $\zeta_1,\ldots, \zeta_{m_{\beta}}\in \mathcal{K}$ such that
 \begin{equation}\label{E:5.11}
 \begin{aligned}
{\cal R}_{m_{\beta}}((0, \wt y, A_0); \zeta_1,\ldots, \zeta_{m_{\beta}})
    & = (0,\wt y, T_{21}(\beta) A_0).
\end{aligned}
\end{equation}
\end{lemma}

\begin{proof}
We first prove that for sufficiently large $m_{\alpha}\ge1$, there is $\zeta_1\in \mathcal{K}$ such that $${\cal R}_1((0, \wt y,  A_0);\zeta_1)=(0, \wt y, T_{12}({\alpha}/{m_\alpha}) A_0).$$
Indeed, similar to the proof of Proposition~\ref{P:5.2}, we choose $g(t)$ as in \eqref{E:5.6} with $a_1=-{\alpha}/{m_{\alpha}}$ in \eqref{E:5.7}. We again define $f$ as in \eqref{E:5.8} and take the shear~flow
$$
   u(t,x) = f(t) \begin{pmatrix}
                       \cos x_2 \\
                       0 \\
                     \end{pmatrix},
$$
which verifies the following properties:
\begin{itemize}
	\item[(a)]  $u(0)= u(1)=0$,
	\item[(b)]  the solution $y(t)=(y_1(t),y_2(t))$ of
	\begin{equation*}
\left\{ \begin{aligned} 
  &  \dot{y}(t) = u(t, y(t)),    \\
    & y(0) =  \wt y
\end{aligned}\right.
\end{equation*}
satisfies $y(t) = \wt y$, $ t \in J$,  
	\item[(c)] $u(t,x)$ is a solution of \eqref{E:5.1} with control   
\begin{align}\label{E:5.12}
 \zeta_1(t,x) :=& (f'(t)+\nu f(t))
        \begin{pmatrix}
                       \cos x_2 \\
                       0 \\
                     \end{pmatrix} 
          =  g(t) 
             \begin{pmatrix}
                       \cos x_2 \\
                       0 \\
                     \end{pmatrix} .
\end{align}
\end{itemize}
Choosing $m_{\alpha}$ sufficiently large, we make the numbers $a_1$ and $a_j$ in \eqref{E:5.7} small enough to guarantee that $\zeta_1\in \mathcal{K}$. 
Now let us consider the corresponding matrix process
\begin{equation*}
  \left\{
  \begin{aligned}
   & \dot{A}(t) = D_xu(t, y(t)) A(t), \\
   & A(0) =  A_0.
  \end{aligned}
  \right.
\end{equation*}
By direct computation, we see that
\begin{align*}
   A(1) = 
     \begin{pmatrix}
       1 & {\alpha}/{m_{\alpha}} \\
       0 & 1 \\
     \end{pmatrix}A_0=T_{12}({\alpha}/{m_{\alpha}}) A_0.
\end{align*}Thus, we proved the exact controllability of the matrix process $A(t)$ from~$A_0$ to $T_{12}({\alpha}/{m_{\alpha}}) A_0$, while keeping $u(t)$ and $y(t)$ unchanged. Note that the choice of $m_\alpha$ does not depend on $A_0$.
 
 Now to prove \eqref{E:5.10}, let us introduce the points $\Upsilon_k=(0, \wt y, T_{12}({k\alpha}/{m_{\alpha}})A_0)$ with $k=1,\ldots, m_{\alpha}$. Applying the above exact controllability property, we find controls $\zeta_2, \dots, \zeta_{m_{\alpha}} \in \mathcal{K}$ such that ${\cal R}_1(\Upsilon_{k-1}; \zeta_k)= \Upsilon_k$ for $2\leq k\leq m_{\alpha}$ (actually, we can take $\zeta_k,$ $k=2,\ldots, m_{\alpha}$ equal to $\zeta_1$ in \eqref{E:5.12}). This proves~\eqref{E:5.10}.

In a similar way, for sufficiently large $m_{\beta}\ge 1$,
we use the shear flow
\begin{align*}
   u(t,x)= f(t) 
        \begin{pmatrix}
          0 \\
          \cos x_1 \\
        \end{pmatrix}
\end{align*}
with $f$ defined by \eqref{E:5.8} and $g$ defined by \eqref{E:5.6} with $a_1=-{\beta}/{m_{\beta}}$ to control from $(0,\wt y,A_0)$ to
$(0,\wt y, T_{21}({\beta}/{m_{\beta}}) A_0)$. In this case, the control generated by this shear flow is
$$\zeta_1(t,x)=g(t)
        \begin{pmatrix}
          0 \\
          \cos x_1 \\
        \end{pmatrix},$$
which   belongs to $\mathcal{K}$ for large enough $m_{\beta}$. Moreover, repeating the above argument, we can find controls $\zeta_2, \ldots,
\zeta_{m_{\beta}}\in \mathcal{K}$ that steer the system from $(0,\wt y, T_{21}({\beta}/{m_{\beta}}) A_0)$ to $(0, \wt y, T_{21}(\beta)A_0)$.
This completes the proof of \eqref{E:5.11}.
\end{proof}

Now we are ready to prove Proposition~\ref{P:5.3} and Theorem~\ref{T:5.1}.

\begin{proof}[Proof of Proposition~\ref{P:5.3}]
From Lemma~\ref{L:5.4} it follows that 
there are transvection matrices
$T_1, \ldots, T_4 $ 
such that
$A^\sharp= T_4\circ\dots \circ T_1  \circ A_0$.
Let us set
$$
B_0:= A_0, \quad B_j:= T_j\circ\dots\circ T_1\circ  A_0, \quad j=1,\dots,4.
$$ 
By Lemma~\ref{L:5.5}, for any $j=1,\ldots,4$,
there is an integer $m_j\ge1$ depending on $B_{j-1}$ and $B_j$, and controls $\zeta^j_1,\ldots, \zeta^j_{m_j}$  steering the system from $(0,\wt y, B_{j-1})$ to~$(0,\wt y, B_j)$.
Thus, taking $m=\sum_{j=1}^4m_j$ and
$$
 (\zeta_1, \ldots, \zeta_{m})= (\zeta_1^1,\ldots, \zeta_{m_1}^1,\ldots,\zeta_1^4,\ldots, \zeta^4_{m_4}),
$$
we obtain the required result \eqref{E:5.9}.
\end{proof}

\begin{proof}[Proof of Theorem~\ref{T:5.1}]
 The proof is divided into three steps. 
 \paragraph{\bf Step 1.}By the dissipativity of the Navier--Stokes system, the velocity field~$u(t)$ of the unforced equation converges to zero in $V^5$. As a result, for any $\ve_1>0$,  sufficiently large $m_1\geq1$, and
 any initial state $(u_0, y_0)\in \mathcal{X}$, we~have
$$
  d_{\wh{\HH}}\left({\cal R}_{m_1}((u_0, y_0, \Id_{\mathbb{R}^2});0,\ldots,0),
 (0, y_1, A_1)\right)<\ve_1
$$for some $y_1\in \bbt^2$ and $ A_1 \in\Sl_2(\R)$.

\paragraph{\bf Step 2.} Since the set of velocities reachable from~$0$ is dense in the compact set~$X$, and since the zero control keeps the velocity at rest, for any $\varepsilon_3>0$ there is an integer $m_3\geq1$ such that any target $u^\sharp\in X$ can be reached from~$0$ within accuracy~$\varepsilon_3$ in exactly $m_3$ steps, with suitable controls $\zeta_1^3,\ldots,\zeta^3_{m_3}\in \mathcal{K}$. Let us fix such a controlled trajectory and a target $(y^\sharp, A^\sharp)\in \bbt^2\times\Sl_2(\R)$, and let $y_2\in\bbt^2$ and $A_2\in\Sl_2(\R)$ be obtained by pulling $y^\sharp$ and~$A^\sharp$ back along the flows \eqref{E:5.2} and \eqref{E:5.3} generated by this trajectory. Then, starting from $(0,y_2,A_2)$ and applying the same controls, the particle and matrix components reach exactly $(y^\sharp,A^\sharp)$ at time~$m_3$, while the velocity component ends within $\varepsilon_3$ of~$u^\sharp$. Hence
\[
d_{\wh{\HH}}\left(
{\cal R}_{m_3}
((0,y_2,A_2);\zeta_1^3,\ldots,\zeta_{m_3}^3),
(u^\sharp,y^\sharp,A^\sharp)
\right)
<\varepsilon_3.
\]
\paragraph{\bf Step 3.} We now steer the point $(0, y_1, A_1)$ obtained in Step 1 to the point $(0, y_2, A_2)$ constructed in Step 2, in three stages: by Proposition~\ref{P:5.2}, some $m_4$ controls $\zeta^{2,1}_1, \ldots, \zeta^{2,1}_{m_4}\in \mathcal{K}$ steer $(0, y_1, A_1)$ to $(0, \tilde{y}, A_1)$; by Proposition~\ref{P:5.3}, some $m_5$ controls $\zeta^{2,2}_{1},\ldots, \zeta^{2,2}_{m_5}\in \mathcal{K}$ steer $(0, \tilde{y}, A_1)$ to $(0, \tilde{y}, A_2)$; and, again by Proposition~\ref{P:5.2}, some $m_6$ controls $\zeta^{2,3}_{1}, \ldots, \zeta^{2,3}_{m_6}\in \mathcal{K}$ steer $(0, \tilde{y}, A_2)$ to $(0, y_2, A_2)$. Thus, setting $m_2:=m_4+m_5+m_6$ and 
$$(\zeta^2_1, \ldots, \zeta^2_{m_2}):=(\zeta^{2,1}_1, \ldots, \zeta^{2,1}_{m_4}, \zeta^{2,2}_1,\ldots, \zeta^{2,2}_{m_5}, \zeta^{2,3}_1,\ldots, \zeta^{2,3}_{m_6}),$$ we obtain 
$$\mathcal{R}_{m_2}((0, y_1, A_1);\zeta^2_1,\ldots, \zeta^2_{m_2})=(0, y_2, A_2).$$
Finally, setting $m:=m_1+m_2+m_3$ and using the continuity of ${\cal R}_m$, we see that~\eqref{E:5.5} holds with the controls 
$$
(\zeta_1,\ldots, \zeta_m)=(0,\ldots,0,\zeta^2_1,\ldots, \zeta_{m_2}^2,  \zeta_1^3,\ldots,  \zeta^3_{m_3}).
$$
\end{proof}
Combining Proposition~\ref{P:5.2} with Steps~1 and~2 of the proof of Theorem~\ref{T:5.1}, we obtain the following stronger version of Hypothesis \hyperlink{H2}{\rm(H$_2$)}.
 \begin{corollary}\label{C:5.6}
For any $\varepsilon>0$, there is an integer $m\geq1$ such that, for any~$z_0, z^\sharp\in \mathcal{X}$ and appropriate controls $\zeta_1,\ldots, \zeta_m\in \mathcal{K}$, we have
$$
d_{\HH}\left(S_m(z_0;\zeta_1,\ldots,\zeta_m),  z^\sharp\right)< \varepsilon.
$$
\end{corollary}

\section{Controllability of the linearized Lagrangian flow}  \label{S:6}

In the present section, we consider the linearized system
\begin{gather} 
     \partial_t v+ \LLLL v + Q(u) v = \zeta, \label{E:6.1} \\
    \dot{q}(t) = (D_xu)(t,y)q + v(t,y),  \label{E:6.2} \\ 
      \dot{B}(t) =   (D^2_{xx}u)(t,y)A q + (D_xu)(t,y) B
   + (D_x v)(t,y) A,  \label{E:6.2*} 
\end{gather} where $\zeta\in E=L^2(J,\eE)$ is a control. The system is supplemented with the initial condition
$$
(v(0),q(0),B(0))=(0,0,0)\in V^5\times {\mathsf T}_{y(0)}\bbt^2\times  {\mathsf T}_{A(0)}\Sl_2(\R).
$$
Here  
\begin{equation}\label{E:Q}
Q(u) v := \Pi(\<u,\na\> v + \<v,\na\> u)
\end{equation} is the linearization of the nonlinear term in \eqref{E:5.1}, 
$(D^2_{xx}u)(t,y)Aq$ is the~matrix defined by  
$$
  (D^2_{xx}u)(t,y)Aq:=q_1(\partial_1D_xu)(t,y) A + q_2(\partial_2 D_xu)(t,y)A,
$$
and $(u(t),y(t), A(t))$ is a solution of the nonlinear system \eqref{E:5.1}-\eqref{E:5.3} with initial data  $(u(0),y(0), A(0))\in\XX \times \{\Id_{\R^2}\}$ and force $\eta \in E$.
  In what follows, we denote  
$$
\widehat{\mathcal{T}}_s:=V^5\times  {\mathsf T}_{y(s)}\bbt^2\times  {\mathsf T}_{A(s)}\Sl_2(\R)
$$ and consider the linear operator 
\begin{gather*}
 \mathbb{A}(u,y,\eta):E\to \widehat{\mathcal{T}}_1, \quad \zeta\mapsto (v(1),q(1), B(1)). 	
\end{gather*}
  This operator is the linearization of the mapping $\eta \mapsto (u(1),y(1), A(1))$. Recall that $\ell$ stands for the law of the random variables~$\{\eta_k\}$ in~\eqref{E:4.1} satisfying Hypothesis~{\rm \hyperlink{N}{(N)}}.
The following theorem is the main result of this~section. It~shows that Hypothesis~{\rm \hyperlink{H7}{(H$_7$)}} is verified.

\begin{theorem}\label{T:6.1}
Under Hypothesis~{\rm \hyperlink{N}{(N)}}, for any  $(u(0),y(0))\in\XX$ and $\ell$-a.e. $\eta\in E$, the image of~$ \mathbb{A}(u,y,\eta)$ is dense in $\widehat{\mathcal{T}}_1$. 
\end{theorem}
This theorem is proved by extending some ideas from Section~4 in~\cite{KNS20} based on the observability assumption and a saturation property of the space~$\eE$.
As the current situation is more degenerate (the control $\zeta$ acts directly only on a few Fourier modes in the velocity equation \eqref{E:6.1}), the verification of the controllability of the linearized triple system~\eqref{E:6.1}-\eqref{E:6.2*} is more~delicate.  

\subsection{Forward and backward linear flows}\label{S:6.1}

Let us consider the following homogeneous version of the system \eqref{E:6.1}-\eqref{E:6.2*}:
\begin{gather}
    \partial_t v + \LLLL v + Q(u) v =0,  \label{E:6.3}\\
  \dot{q}(t) = (D_xu)(t,y) q + v(t,y),    \label{E:6.4} \\ 
    \dot{B}(t) =  (D^2_{xx}u)(t,y)A q + (D_xu)(t,y) B 
   + (D_x v)(t,y) A,  \label{E:6.4*} 
\end{gather} 
  where $(v(s), q(s), B(s)) = (v_0, q_0, B_0)\in  V^5 \times \bbr^2
\times \M_2(\R)$. Let $R(t,s)$, $0\leq s\leq t\leq 1$, denote the
resolving operator of this system, acting in the space $V^5 \times
\bbr^2 \times \M_2(\R)$, so that
$(v(t),q(t),B(t))=R(t,s)(v_0,q_0,B_0)$. Note that
$
R(t,s): \widehat{\mathcal{T}}_s \to \widehat{\mathcal{T}}_t.
$ 
Additionally, let us set~$\Lambda:=\left(-\Delta\right)^\frac12$ and introduce the  following backward system:
\begin{gather}
   \partial_t w - \LLLL w - \Lambda^{-10} Q^*(u) \Lambda^{10}w  + \Lambda^{-10} F(p,U, y,A) =0,  \label{E:6.5}\\
   \dot{p}(t) + (D_xu)^\TTT(t,y)p 
   +   (D^2_{xx}u (t,y) A)^\TTT U=0,   \label{E:6.6} \\ 
   \dot{U}(t) + (D_xu)^\TTT(t,y) U=0, \label{E:6.6*}
\end{gather}
where 
$$
(w(1), p(1), U(1)) = (w_0, p_0, U_0) \in  V^5 \times \bbr^2 \times \M_2(\R),
$$ 
$Q^*(u)$ is the (formal) $H$-adjoint of $Q(u)$ 
given by 
$$
 Q^*(u) w = -\Pi\left(\lag u, \na \rag  w + (\na \otimes w) u\right) 
$$
with $\na\otimes w$ being the $2\times 2$ matrix 
with entries $\partial_i w_j$, 
$(D_xu)^\TTT(t,y(t))$ is the transpose matrix of 
$(D_xu)(t,y(t))$,   $(D^2_{xx}u (t,y) A)^\TTT U $ is the vector with components   
$$ 
 ((D^2_{xx}u (t,y) A)^\TTT U)_i=\lag (\partial_iD_xu)(t,y)A, U \rag_{\M_2},\quad i=1,2,
$$ 
 and $F(p,U, y, A)$ is defined by duality as follows. 
For any fixed $\sigma\in  (2,5/2)$ and any given 
$$
(p,U, y, A)\in  \mathbb{R}^2 \times \M_2(\R) \times \bbt^2 \times \Sl_2(\R),
$$ we~define~$F(p,U,y,A)$ to be the unique element of $V^{-\sigma}(\bbt^2)$ such that  
\begin{equation} \label{E:6.7}
   {}_{V^{\sigma}}\<\xi, F(p,U, y, A) \>_{V^{-\sigma}}
   =  \<\xi(y), p\>_{\bbr^2} 
      + \<(D_x\xi)(y)A, U\>_{\M_2}
\end{equation}for any $\xi\in V^\sigma$.
The Cauchy--Schwarz inequality and the Sobolev embedding
$V^\sigma(\bbt^2) \hookrightarrow C^1(\bbt^2)$ imply that,
with $C_\sigma$ denoting the norm of this embedding,\footnote{In what follows, we write $\|\cdot\|_{\R^2}$ and $\|\cdot\|_{\M_2}$ for the norms in $\R^2$ and $\M_2(\R)$.}
\begin{align*}
   |\<\xi(y), p\>_{\bbr^2}| + |\<D_x\xi(y)A, U\>_{\M_2}|
   &\leq \|\xi(y)\|_{\bbr^2} \|p\|_{\bbr^2}  
         + \|D_x \xi(y) \|_{\M_2} \|A\|_{\M_2} 
         \|U\|_{\M_2} \\ 
&\leq  C_\sigma\|\xi\|_{V^\sigma} 
          \left(\|p\|_{\bbr^2} + \|A\|_{\M_2}\|U\|_{\M_2}\right),
\end{align*}
which shows that
$F(p,U,y,A) \in V^{-\sigma}$ is well defined   
and 
$$
 \|F(p,U,y,A)\|_{V^{-\sigma}} 
  \leq C_\sigma(\|p\|_{\mathbb{R}^2} 
       +  \|A\|_{\M_2}\|U\|_{\M_2}). 
$$ 
Consequently, for any $p\in C(J, \bbr^2)$, $U\in C(J, \M_2(\R))$, $y\in C(J, \bbt^2)$, and $A\in C(J, \Sl_2(\R))$, 
we have 
\begin{equation} \label{E:6.8}
   \|F(p,U, y, A)\|_{L^\infty(J, V^{-\sigma})}
   \leq C_\sigma(\|p\|_{C(J, \bbr^2)} 
     +  \|A\|_{C(J, \M_2(\R))}\|U\|_{C(J, \M_2(\R))}).
\end{equation}

\begin{remark}\label{R:adjoint}
The system \eqref{E:6.5}-\eqref{E:6.6*} is the formal adjoint of
the system \eqref{E:6.3}-\eqref{E:6.4*} with respect to the inner
product of $V^5\times\bbr^2\times \M_2(\R)$. Indeed, the
equations \eqref{E:6.6} and \eqref{E:6.6*} for $(p,U)$ are the
adjoints of the ODEs \eqref{E:6.4} and \eqref{E:6.4*} for
$(q,B)$, while the coupling terms $v(t,y)$ and $(D_xv)(t,y)A$ in
the latter give rise to the distribution-valued term
$F(p,U,y,A)\in V^{-\sigma}$ in the $w$-equation \eqref{E:6.5};
cf.~the duality relation \eqref{E:6.7}. The weights
$\Lambda^{\pm10}$ appear because $Q^*(u)$ is the $H$-adjoint of
$Q(u)$, whereas the duality is taken in~$V^5$:
$\<Q(u)v,w\>_{V^5}=\<v,\Lambda^{-10}Q^*(u)\Lambda^{10}w\>_{V^5}$
and $\<v,\Lambda^{-10}F\>_{V^5}
={}_{V^{\sigma}}\<v,F\>_{V^{-\sigma}}$ for $v\in V^{\sigma}$.
See Lemma~\ref{L:6.3} for the precise statement.
\end{remark}
The following lemma establishes the well-posedness of the system \eqref{E:6.5}-\eqref{E:6.6*}; its~proof is postponed to Section~\ref{S:7.3}.
\begin{lemma}\label{L:6.4}
For any $(w_0, p_0, U_0) \in V^5 \times \bbr^2 \times \M_2(\R)$,
there is a unique~solution $(w,p, U)\in  C(J, V^5 \times \bbr^2 \times \M_2(\R))$
to the system \eqref{E:6.5}-\eqref{E:6.6*}.  
Moreover,  
\begin{align}  \label{E:6.12}
   \|w\|_{C(J, V^5)}
   +  \|w\|_{L^2(J, V^6)}
   &+ \|p\|_{C(J, \bbr^2)} 
   + \|U\|_{C(J, \M_2(\R))}
   \nonumber\\&\lesssim \|w_0\|_{V^5} + \|p_0\|_{\bbr^2} 
   + \|U_0\|_{\M_2} .
\end{align}
\end{lemma}

Let $R(t,s)^*$ be the dual operator of $R(t,s)$ in $ V^5 \times \bbr^2 \times \M_2(\R)$. The next lemma shows that the backward system  \eqref{E:6.5}-\eqref{E:6.6*}
 is actually the dual of the forward system \eqref{E:6.3}-\eqref{E:6.4*}.

\begin{lemma}\label{L:6.3}
For any $f_0:=(w_0,p_0, U_0) \in  V^5\times \bbr^2 \times \M_2(\R)$,
we have that 
$$
(w(t),p(t), U(t))=R(1,t)^*f_0, 
\quad 0\leq t\leq 1
$$ 
is the solution of the backward system \eqref{E:6.5}-\eqref{E:6.6*}.
\end{lemma}

\begin{proof}
Let 
$$
J(t,\tau): V^5 \times \bbr^2 \times \M_2(\R) \to V^5 \times \bbr^2 \times \M_2(\R), \quad
 0\leq \tau\leq t\leq 1
 $$
be the resolving operator for the backward system \eqref{E:6.5}-\eqref{E:6.6*} with final condition at time~$t$: $w(t) = w_0$, $p(t) = p_0$, $U(t) = U_0$.
By virtue of Lemma~\ref{L:6.4},
we have that $J(t,\tau)$ is a bounded operator in $V^5 \times \bbr^2 \times \M_2(\R)$
and   
\begin{equation} \label{E:6.9}
  \partial_\tau J(t,\tau) \in C([0,t], V^3\times \bbr^2 \times \M_2(\R)).
\end{equation}
Let us show that, for any  $h,g\in V^5\times \bbr^2 \times \M_2(\R)$ and $0\leq s\leq \tau\leq t\leq 1$, the~quantity
$$
   \lag R(\tau,s) h, J(t,\tau) g\rag_{V^5\times \bbr^2 \times \M_2} 
$$
is independent of $\tau$.
Once this is proved,
evaluating at $\tau =s$
and $\tau=t$,
we~get
$$
  \<h, J(t,s)g\>_{V^5\times \bbr^2 \times \M_2}
  = \<R(t,s)h,g\>_{V^5\times \bbr^2 \times \M_2},
$$
which yields
$J(t,s) = R(t,s)^*$.
To prove the independence of $\tau$,
we~set  
$$
\mathbb{D}:= \text{diag}(\Lambda^5, \Id_{\R^2}, \Id_{\M_2})
$$
and take the derivative in $\tau$:
\begin{align*}
       \partial_\tau\<R(\tau, s) h , J(t,\tau)g\>_{V^5\times \bbr^2 \times \M_2}
  & =       \< \mathbb{D} \partial_\tau R(\tau, s) h ,  \mathbb{D} J(t,\tau)g\>
      \\
  &\quad+  \<  \mathbb{D} \partial_\tau  J(t,\tau)g, 
      \mathbb{D} R(\tau, s) h\>
\end{align*} 
for $0\leq s< \tau < t\leq 1$, 
where $\<\cdot, \cdot\>$ 
stands for the duality between 
${V^{-1}\times \bbr^2 \times \M_2(\R)}$ 
and ${V^1\times \bbr^2 \times \M_2(\R)}$. 
Note that,
by the smoothing property of $R(\tau, s)$ and~\eqref{E:6.9},
the right-hand side above is well-defined.

Let us set
$\mathbb{F}_1:=(0,
       v(\tau,y),  
       (D_xv)(\tau,y) A)^\TTT$, 
$\mathbb{F}_2:= (\Lambda^{-10} F(p,U,y,A), 
       0, 
       0)^\TTT$, 
$$\mathbb{M}_1 := \begin{pmatrix}
       -\LLLL - Q(u) &  0 & 0\\
       0& (D_xu)(\tau,y) & 0 \\
       0& (D^2_{xx}u)(\tau,y) A & (D_xu) (\tau,y)\\ 
     \end{pmatrix},$$
and 
$$\mathbb{M}_2:= \begin{pmatrix}
       \LLLL+   \Lambda^{-10}Q^*(u) \Lambda^{10}& 0 & 0  \\
      0&   -(D_xu)^\TTT(\tau,y) &  
      -((D^2_{xx}u)(\tau,y)A)^\TTT  \\ 
      0 & 0 & -(D_xu)^\TTT(\tau,y) \\ 
     \end{pmatrix}.
$$
Using the expressions of the derivatives $\partial_\tau R(\tau, s)$ and $\partial_\tau  J(t,\tau)$ from the 
systems \eqref{E:6.3}-\eqref{E:6.4*}
and  \eqref{E:6.5}-\eqref{E:6.6*},
we get
\begin{align*}
     \partial_\tau     \<R(\tau, s) h , J(t,\tau)g\>_{V^5\times \mathbb{R}^2 \times \M_2} &   
   =    \left\lag 
     \mathbb{D}^2 \mathbb{M}_1
    R(\tau,s) h
    + \mathbb{D}^2 \mathbb{F}_1, 
    J(t,\tau) g \right\rag_{H\times\bbr^2\times \M_2}   \\
    &\quad
    +    
    \left\lag
  \!  R(\tau, s)h, 
     \mathbb{D}^2 \mathbb{M}_2 J(t,\tau) g
   \!-\!
     \mathbb{D}^2 \mathbb{F}_2
   \right\rag_{H\times\bbr^2\times \M_2} \\
&=  \<v(\tau,y), p\>_{\bbr^2} 
   + \<(D_xv)(\tau,y)A, U\>_{\M_2} \\
   &\quad  - {}_{V^{\sigma}}\<v, F(p,U, y, A)\>_{V^{-\sigma}}
   = 0,
\end{align*}
 where in the last step we used \eqref{E:6.7}.
This completes the proof.
\end{proof}

\subsection{Approximate controllability via saturation}\label{S:6.2}

 Let us define recursively a non-decreasing sequence of finite-dimensional spaces $\eE_k$  in the following way:
 \begin{itemize}
  \item $\eE_0=\eE$,
  \item if $\eE_k$ is defined, then $\eE_{k+1}$ is the space spanned by the vectors  
      $$
      \eta+\sum_{l=1}^nQ(\zeta_l)\xi_l,\quad n\ge1, 
      $$
      where $\eta, \zeta_l\in \eE_k$ and~$\xi_l\in \eE$, $l=1,\ldots,n$ and $Q$ is defined by \eqref{E:Q}.
\end{itemize}
The following lemma is a consequence of the results established in  Section~6 in \cite{KNS20}.

\begin{lemma}\label{L:6.5}
 The space $\eE$  is saturating, that is, the union $\bigcup_{k\geq 0} \eE_k$ is dense in $V^m$ for any $m\ge1$.
\end{lemma}

Now, we are ready to prove Theorem \ref{T:6.1}.

\begin{proof}[Proof of Theorem \ref{T:6.1}]
The proof develops the ideas of Section 4.1 in \cite{KNS20}. 
Let  
$
\overline{{\eE}} := \eE \times \{0\}\times \{0\},
$ and let $\mathbb{A}= \mathbb{A}(u,y,\eta): L^2(J, \overline{\eE})\to \widehat{\mathcal{T}}_1$ be the resolving operator of the system \eqref{E:6.1}-\eqref{E:6.2*} with extended control $\ol{\zeta}$:  
$$
\ol{\zeta} = (\zeta, 0, 0) \in L^2(J, \ol{\eE}) \mapsto (v(1), q(1), B(1))\in \widehat{\mathcal{T}}_1.
$$
In what follows, we assume that $\eta\in E$ is observable.
By the Duhamel formula,
$$
    \mathbb{A}\ol{\zeta} = \int_0^1 R(1,s) \ol{\zeta}(s) \dd s.
$$
We need to show that the image of $\mathbb{A}$ is dense in $\widehat{\mathcal{T}}_1$. This will be achieved by showing that
  the kernel of $\mathbb{A}^*$ (the dual of $\mathbb{A}$) is trivial.
The operator $\mathbb{A}^*$ is given explicitly by
$$
   \mathbb{A}^* f(t) = {\mathsf P}_{\ol{\eE}}  R(1,t)^* f, \ \   t\in J,\   f\in \widehat{\mathcal{T}}_1,
$$
where ${\mathsf P}_{\ol{\eE}}$ is the orthogonal projection onto $\ol{\eE}$. Let $f_0 = (w_0, p_0, U_0)$ be any element of $\textup{Ker}(\mathbb{A}^*)$, i.e.,
$$
   \mathbb{A}^*f_0 = 0\ \ \text{in}\ L^2(J, \overline{\eE}).
$$
This implies that, for any $\ol{\xi} = (\xi,0,0) \in \ol{\eE}$, we have
\begin{equation}  \label{E:6.13}
  \<\ol{\xi}, R(1,t)^*f_0\>_{V^5 \times \bbr^2\times \M_2} =0\ \ \text{for  a.e.}\ t\in J.
\end{equation}
As $R(1,t)^* f_0$ is continuous in $t$, the previous
 equality holds for any $t\in J$.
In~particular,
taking $t=1$, we get that $w_0 \perp \eE$. 

Based on this information,
let us show that the equality
\eqref{E:6.13} holds for any $\ol{\xi} = (\xi, 0, 0) \in \ol{\eE}_k:= \eE_k \times \{0\} \times \{0\}$, any $k\geq 0$, 
and $t\in J$.  For this purpose,
we shall use an inductive argument. The base case $k=0$ has already been considered.
Suppose that there exists $k\geq 0$ 
such that \eqref{E:6.13} holds for any~$ \ol{\xi} = (\xi, 0,0) \in \ol{\eE}_k$ and $t\in J$.  
Taking the derivative in time in~\eqref{E:6.13}
 and using Lemma~\ref{L:6.3},
we~get 
\begin{gather*}
   \bigg<
        \begin{pmatrix}
         \Lambda^5 \xi \\
          0   \\ 
          0 \\
        \end{pmatrix},
          \begin{pmatrix}
       \Lambda^5  \LLLL+ \Lambda^{-5} Q^*(u) \Lambda^{10} & 0 & 0 \\
       0 & -(D_xu)^\TTT(t,y) 
       & -(D^2_{xx}u \, A)^\TTT(t,y) \\
       0 &  0& -(D_xu)^\TTT(t,y)
        \end{pmatrix}
   f  \bigg>  \\ 
   - \bigg< \begin{pmatrix}
          \xi \\
          0   \\ 
          0 \\
        \end{pmatrix},
        \begin{pmatrix}
       F(p,U,y,A) \\
       0 \\ 
       0 \\ 
        \end{pmatrix}
   \bigg>  = 0\ \
\end{gather*} 
for all $t\in J$, where $\<\cdot, \cdot\>$ 
denotes the duality between 
${V^{\sigma}\times \bbr^2 \times \M_2(\R)}$ 
and ${V^{-\sigma}\times \bbr^2 \times \M_2(\R)}$,
and 
$f(t) = R(1,t)^*f_0 = (w(t), p(t), U(t))$.
This is equivalent to
$$
   \<\LLLL\xi  + Q(u) \xi, w(t)\>_{V^5}
    -{}_{V^\sigma}\<\xi, F(p, U, y, A)\>_{V^{-\sigma}} =0, \ \   t\in J.
$$
We write
$$
   \wt u(t):
   =  u(t) - \int_0^t \eta(s) \dd s
   =  u(t) - \sum\limits_{j\in \mathcal{I}} e_j\int_0^t \eta^j(s)\dd s,
$$
where $\wt u(t)$ is smoother in time than $u(t)$, and 
$\eta^j(s) := \<\eta(s), e_j\>_{V^5}$. 
Then, taking into account the relation \eqref{E:6.7},
we derive
\begin{gather}
      \sum\limits_{j\in \mathcal{I}}
      \<Q(e_j) \xi, w(t)\>_{V^5} \int_0^t \eta^j(s) \dd s
    + \<\LLLL\xi + Q(\wt u) \xi, w(t)\>_{V^5} \nonumber \\ 
      - \<\xi(y(t)), p(t)\>_{\bbr^2} 
    - \<(D_x\xi)(y(t))A(t), U(t)\>_{\M_2} 
    =0  \nonumber 
\end{gather} 
for all $t\in J$.
Taking the derivative in this equality, we get
$$
   \sum\limits_{j\in \mathcal{I}} a_j(t) \eta^j(t) + b(t) =0,
$$
where
\begin{align*}
    a_j(t):=& \<Q(e_j) \xi, w(t)\>_{V^5}, \\
    b(t) : =& 
      \sum\limits_{j\in \mathcal{I}}
     \frac{\dd}{\dd t} \<Q(e_j) \xi, w(t)\>_{V^5} 
     \int_0^t \eta^j(s) \dd s 
   +  \frac{\dd}{\dd t} 
     \<\LLLL\xi + Q(\wt u(t)) \xi, w(t)\>_{V^5}  \\ 
   & -   \frac{\dd}{\dd t} \<\xi(y(t)), p(t)\>_{\bbr^2} 
      -   \frac{\dd}{\dd t} \<(D_x\xi)(y(t))A(t), U(t)\>_{\M_2}.
\end{align*}
Note that, $a_j$ is differentiable and
$b(t)$ is continuous on $J$, 
due to the fact that ~$\xi$ and $e_j$ are smooth,   
$w\in C(J, V^5)$, 
$\partial_t w\in C(J, V^3)$,  
$\tilde u\in C(J, V^8)$, 
and $y$, $U$, $p$, $A$ are continuously differentiable.  
By the observability of $\eta$,
it follows that 
$$
   a_j(t) = \<Q(e_j) \xi, w(t)\>_{V^5} = 0,\ \ j\in \mathcal{I},\ \     t\in J
$$for any~$\xi \in \eE_k$ 
and~$e_j\in \eE$.
In particular, $w(t) \perp \eE_{k+1}$, $t\in J$. Thus, by induction, $w(t)$ is orthogonal to all $\eE_k$, $k\geq 0$, 
and so \eqref{E:6.13} holds 
for any~$k\geq 0$ and $t\in J$, 
as claimed. 

Since $\bigcup_{k\geq 0} \eE_k$ is dense in~$V^5$,
it follows that
$ w(t) =0$, $t\in J$.
In particular, $w_0 =w(1)=0$. This, along with the equation \eqref{E:6.5} of $w$, yields 
\begin{equation*}
    F(p(t),U(t), y(t), A(t)) = 0, \ \ t\in J,
\end{equation*}
which leads to (cf. \eqref{E:6.7})
\begin{equation*}
    \<\xi(y(t)), p(t)\>_{\mathbb{R}^2} 
    +  \<D_x\xi(y(t))A(t), U(t)\>_{\M_2} 
    = 0,\ \   t\in J,\ \   \xi \in V^\sigma.
\end{equation*}
As in this equality $\xi\in V^\sigma$ is arbitrary,
we infer that
$$
   p(t) =0,\ \  U(t)=0,\ \ t\in J.
$$
Evaluating at $t=1$, we get $p_0=0$, 
$U_0 =0$. 
Therefore, we conclude that
$f_0=(w_0, p_0, U_0)=0$,
that is, the kernel of $\mathbb{A}^*$ is trivial.
This completes the proof of the theorem.
\end{proof}

\medskip

\appendix
\refstepcounter{section}
\section*{Appendix}

\subsection{Proof of Theorem~\ref{T:2.4}}\label{S:7.1}

 We adopt the approach used in the proof of Proposition~3.4 in~\cite{CR24}. The~following lemma is an abstract version of Proposition~3.3 in~\cite{CR24}. In what follows, we denote  $\wh Z:=Z\times\mathcal{A}(\Omega\times Z)$, noting that, under the assumptions of Theorem~\ref{T:2.4}, 
$\wh Z$ is a compact metric space. 
 \begin{lemma}\label{continuous-measure}
Under the assumptions of Theorem~\ref{T:2.4}, there exists a continuous family of probability measures 
$\nu: \wh Z \to \mathcal{P}(P^{d-1})$
such that 
\begin{equation}\label{invariant}
\nu_{z_0, A_0}=(A_0)_*^{-1}\int \nu_{z,A}\wh{P}_1(z_0, \dd z, \dd A),
\end{equation}
where $\mathcal{P}(P^{d-1})$ is endowed with the dual-Lipschitz metric.
\end{lemma}
\begin{proof}
For any $(z,A)\in \wh Z$, we define recursively a sequence of measures $\{\nu_{(z,A)}^n\} \subset \mathcal{P}(P^{d-1})$ as follows: $\nu_{(z,A)}^0$ is the uniform probability measure on~$P^{d-1}$, and 
$$
 \nu_{(z,A)}^{n+1}:=A_*^{-1}\int \nu_y^{n} \wh{P}_1(z,  \dd y).
$$
 Let us show that there exists a constant $M>0$, independent of $n\ge0$, such~that \begin{equation}\label{Lipschitz-continuous}
\|\nu_{\tilde{z}_1}^n-\nu_{\tilde{z}_2}^n\|_{L(P^{d-1})}^*\le M \, d_{\wh{Z}}(\tilde{z}_1, \tilde{z}_2)^{\gamma}
\end{equation} for any $\tilde{z}_1=(z_1, A_1), \tilde{z}_2=(z_2, A_2)\in \wh Z$, where $\gamma$ is the number in~\eqref{E:ineq1}.   
 For this purpose, first note that the inequality \eqref{Lipschitz-continuous} holds for $n=0$. Assuming that it holds for $n\ge0$, let us prove it for $n+1$. We have  
\begin{align}\label{es1}
\notag\|&\nu_{\tilde{z}_1}^{n+1}-\nu_{\tilde{z}_2}^{n+1}\|_{L(P^{d-1})}^*=\sup_{\|f\|_L\le1}\left|\langle f, \nu_{\tilde{z}_1}^{n+1} \rangle - \langle f, \nu_{\tilde{z}_2}^{n+1}\rangle\right| \\ \notag
&=\sup_{\|f\|_L\le1} \left|\langle f, (A_1)_*^{-1}\int \nu_y^n\wh{P}_1(z_1, \dd y)\rangle
- \langle f, (A_2)_*^{-1}\int \nu_y^n\wh{P}_1(z_2, \dd y)\rangle\right|\\ \notag
&=\sup_{\|f\|_L\le1}\left|\int \langle f\circ A_1^{-1}, \nu^n_y\rangle \wh{P}_1(z_1, \dd y)-\int \langle f\circ A_2^{-1}, \nu_y^n\rangle \wh{P}_1(z_2,\dd y)\right|\\ \notag
&\le\sup_{\|f\|_L\le1}\left|\int \langle f\circ A_1^{-1}-f\circ A_2^{-1},\nu_y^n\rangle \wh{P}_1(z_1, \dd y)\right|\\  
&\quad+ \sup_{\|f\|_L\le1}\left|\int \langle f\circ A_2^{-1}, \nu_y^n\rangle \wh{P}_1(z_1, \dd y)
- \int \langle f\circ A_2^{-1}, \nu_y^n\rangle \wh{P}_1(z_2, \dd y)\right| .
\end{align}
As $\mathcal{A}(\Omega\times Z)$ is compact, there are constants $C_0, C_1>0$  such that 
\begin{align}\label{es2}
|f(A_1^{-1}v)-f(A_2^{-1}v)|&\le  C_0d_{P^{d-1}}(A_1^{-1}v, A_2^{-1}v)\nonumber\\&\le C_1|A_1-A_2|\le C_1 d_{\wh{Z}}(\tilde{z}_1, \tilde{z}_2)
\end{align}for any $f\in L(P^{d-1})$ with $\|f\|_L\le1$, $A_1, A_2\in \mathcal{A}(\Omega\times Z)$, and $v\in P^{d-1}$. Let~us 
define $g(y):=\langle f\circ A_2^{-1}, \nu_y^n\rangle$ and note that $\|g\|_\ty\le1$. Furthermore, as \eqref{Lipschitz-continuous} holds for~$n$, there exists a constant~$C_2>0$ such that 
\begin{align*}
\notag|g(y_1)-g(y_2)|&=|\langle f\circ A_2^{-1}, \nu_{y_1}^n-\nu_{y_2}^n\rangle|\\ 
&\le \|f\circ A_2^{-1}\|_{L(P^{d-1})}\,\|\nu_{y_1}^n-\nu_{y_2}^n\|_{L(P^{d-1})}^* \nonumber\\&\le C_2\, M\, d_{\wh{Z}}(y_1, y_2)^{\gamma}
\end{align*}for any $y_1, y_2\in \wh{Z}$.
Therefore, $g\in C^{\gamma}(\wh{Z})$ with $|g|_{\gamma}\le C_2M+1$. Taking into account the assumption \eqref{E:ineq1}, let $c$ in  \eqref{E:ineq1} be small enough such that
$$
c\le \frac{C_1D}{2C_1C_2D+1}\le \frac{C_1D+C}{2C_2(C_1D+C)+1},
$$where  $D:=\max_{\tilde{z}_1, \tilde{z}_2 \in \wh Z} d_{\wh{Z}}(\tilde{z}_1, \tilde{z}_2)^{1-\gamma}$. Then, combining this with the inequalities \eqref{es1} and \eqref{es2} and the assumption \eqref{E:ineq1}, we derive 
\begin{align*}
\|\nu_{\tilde{z}_1}^{n+1}-\nu_{\tilde{z}_2}^{n+1}\|_{L(P^{d-1})}^*& \le C_1d_{\wh{Z}}(\tilde{z}_1, \tilde{z}_2)\nonumber\\&\quad+\sup_{\|f\|_L\le 1}\left|\int \wh{P}_1(z_1, \dd y) g(y)-\int \wh{P}_1(z_2, \dd y)g(y)\right|\\ \notag
& \le C_1d_{\wh{Z}}(\tilde{z}_1, \tilde{z}_2)+\left(C+c(C_2M+1)\right)d_{Z}(z_1, z_2)^{\gamma}\\\notag
&\le C_1d_{\wh{Z}}(\tilde{z}_1, \tilde{z}_2)+\left(C+c(C_2M+1)\right)d_{\wh{Z}}(\tilde{z}_1, \tilde{z}_2)^{\gamma} \\ \notag
&\le \left(C_1D+C+c(C_2M+1)\right)d_{\wh{Z}}(\tilde{z}_1, \tilde{z}_2)^{\gamma}\\
&\le M d_{\wh{Z}}(\tilde{z}_1, \tilde{z}_2)^{\gamma},
\end{align*}  
provided that $M= 2(C_1D+C)$. By induction, we get that \eqref{Lipschitz-continuous} holds for any $n\ge0$. 
 
Applying the Arzelà--Ascoli theorem to the sequence $\sigma_{\tilde z}^n:=\frac1n\sum_{j=1}^n \nu_{\tilde z}^j$, we~find a subsequence~$n_j$ along which $\sigma_{\tilde z}^{n_j}$ converges to a continuous map $\nu: \wh Z \to \mathcal{P}(P^{d-1})$. Finally, from the construction it follows that $\nu$ satisfies~\eqref{invariant}.
\end{proof}
The derivation of Theorem~\ref{T:2.4} from the preceding lemma follows the same argument as the derivation of Proposition 3.4 from Propositions~3.1 and~3.3 in~\cite{CR24}. 
\begin{proof}[Proof of Theorem~\ref{T:2.4}]
	 Let us consider the transition function 
$$
P^{\wh Z}(z,A, \Gamma):=\wh P(z, \Gamma), \quad (z,A)\in\wh Z,\,\, \Gamma \in \mathcal{B}(\wh Z),
$$and note that $P^{\wh Z}(z,A, \Gamma)$ is independent of $A$ and the measure $ P^{\wh Z} (\mu\times \delta_{\Id_{\R^d}})$ defined~by
 $$
 P^{\wh Z} (\mu\times \delta_{\Id_{\R^d}})(\Gamma):=\int_Z P^{\wh Z}(z,\Id_{\R^d}, \Gamma) \mu (\dd z)
 $$ is the unique stationary measure for $P^{\wh Z}$, where $\delta_{\Id_{\R^d}}$ is the Dirac measure concentrated at $\Id_{\R^d}$. Let $\nu: \wh Z \to \mathcal{P}(P^{d-1})$ be as in Lemma~\ref{continuous-measure}. The conditions of Proposition~3.1 in~\cite{CR24} are satisfied for the transition function~$P^{\wh Z}$ and the cocycle on~$\wh Z$ defined by the projection $(z,A)\mapsto A$. Therefore, by Proposition~3.1 in~\cite{CR24}, there holds
\begin{equation}\label{E:Lede}
	(A_0)_* \nu_{z_0,A_0}=\nu_{z_1,A_1}, \quad  P^{\wh Z}(z_0,A_0, \dd z_1, \dd A_1) P^{\wh Z} (\mu\times \delta_{\Id_{\R^d}})(\dd z_0,\dd A_0)\text{-a.s.}
\end{equation} In view of the equation~\eqref{invariant}, the measure $(A)_* \nu_{z,A}$ is independent of $A$; we~define $
\nu_z:=(A)_* \nu_{z,A}
$ and note that $\nu:  Z \to \mathcal{P}(P^{d-1})$ is weakly continuous. Then, \eqref{E:Lede} and the continuity of $\nu$ and $\aA$ imply that
$$
\left(\aA^1_{\bomega, z}\right)_* \nu_{z}=\nu_{ {\varphi^1_{\bomega}}z}
$$for any $z\in \supp \mu$ and $\bpP$-a.e. $\bomega$. Iterating this, we get the required~result.
\end{proof}

\subsection{Proof of Lemma~\ref{L:6.4}}\label{S:7.3}

 The well-posedness of the linear equations \eqref{E:6.6} and \eqref{E:6.6*} follows from the classical ODE theory.
In order to study the backward equation~\eqref{E:6.5}, we~replace $1-t$ by $t$ and denote $\Lambda^{5} w(1-t)$ by $w(t)$, 
$u(1-t)$ by~$u(t)$, and similarly for $p, U,y,$ and $A$. Thus we reduce the problem to the well-posedness~of the forward equation
\begin{equation}  \label{E:7.17}
  \partial_t  w + \LLLL  w + \Lambda^{-5} Q^*( u) \Lambda^5  w - \Lambda^{-5} F(p,U,y,A) =0
\end{equation}
with initial data $ w(0) = \Lambda^{5} w_0\in H$
and the a priori estimate
\begin{equation}  \label{E:7.18}
   \| w\|_{C(J, H)} + \| w\|_{L^2(J, V^1)}
   \lesssim \|w(0)\|_{H} 
   + \|F( p,U, y,A)\|_{L^\infty(J, V^{-\sigma})}.
\end{equation}
Below we focus on the estimate \eqref{E:7.18},
as the existence and uniqueness of solutions to \eqref{E:7.17}
can then be derived by standard arguments.

To prove \eqref{E:7.18},  
we let
$g:= \Lambda^{-5} Q^*( u) \Lambda^5  w$ 
and derive from \eqref{E:7.17} the energy equality
\begin{align}  \label{E:7.19.0}
    \frac 12 \| w(t)\|^2_{H}+ \nu \int_0^t \|\na  w\|_{H}^2 \dd s
  &  =  \frac 12 \| w(0)\|^2_{H} 
      - \int_0^t \< w, g\>_H \dd s \nonumber \\ &\quad+ \int_0^t \< w,\Lambda^{-5}F( p,U, y,A)\>_H \dd s. 
\end{align} 
Note that, for any $f\in H$,
\begin{align*} 
    |\<g,f\>_H| 
     =|\<\Lambda  w, \Lambda^4 Q  ( u) \Lambda^{-5} f\>_H|  
     \lesssim \|w\|_{V^1} 
          \|\nabla^4 Q  ( u) \Lambda^{-5} f\|_{H}. 
\end{align*} 
Using the fact that 
\begin{align*}
     \|\na^\alpha \Lambda^{-5}\|_{\mathcal{L}(H,H)} <\infty,\ \ \forall |\alpha| \leq 5 
\end{align*}
and the Sobolev embedding
$H^2(\mathbb{T}^2)  \hookrightarrow L^\infty(\mathbb{T}^2),$
we derive  
\begin{align*}
    \|\nabla^4 Q  ( u) \Lambda^{-5} f\|_{H} 
   \lesssim &  \|\nabla^5 u\|_{H} 
            \|\Lambda^{-5} f\|_{L^\infty} 
            + 
             \|\nabla^4 u\|_{H} 
            \|\na \Lambda^{-5} f\|_{L^\infty}  \\ 
           & + \sum\limits_{j=0}^3 
             \|\nabla^j u\|_{L^\infty} 
             \|\nabla^{5-j} \Lambda^{-5} f\|_{H}  \\ 
    \lesssim & \|u\|_{V^5} \|f\|_{H}. 
\end{align*}
Thus, it follows that 
$$
  |\<g,f\>_H|    \lesssim \|w\|_{V^1} \|u\|_{V^5} \|f\|_{H}, \quad f\in H,          
$$ 
which yields  
\begin{equation}  \label{E:7.19}
   \|g\|_{H} \lesssim \| u\|_{V^5} \| w\|_{V^1}.
\end{equation} 
Moreover, since $\sigma<3$, the operator
$\Lambda^{-6+\sigma}$ is bounded in $H$ and  
\begin{align}  \label{E:7.19*}
      |\< w,\Lambda^{-5} F( p,U, y,A)\>_H| 
     =&  |\< \Lambda w,\Lambda^{-6+\sigma} \Lambda^{-\sigma}F( p,U, y,A)\>_H| \nonumber \\ 
     \leq& \|w\|_{V^1} \|F(p,U, y,A)\|_{V^{-\sigma}}. 
\end{align}
Plugging \eqref{E:7.19} and \eqref{E:7.19*} into \eqref{E:7.19.0}, we get that 
\begin{align*}
     \| w(t)\|_{H}^2+  2\nu \int_0^t \|\na  w\|_{H}^2 \dd s
     &\leq \|w(0)\|_{H}^2 
               + C\int_0^t \| w\|_{H} \| u\|_{V^5} \| w\|_{V^1} \dd s  \\
             & \quad + C\int_0^t \| w\|_{V^1} \|F(p,U, y,A)\|_{V^{-\sigma}} \dd s,  
\end{align*} 
which implies that
\begin{align*}
     \| w(t)\|_{H}^2+\nu
     \int_0^t \| w\|^2_{V^1} \dd s
      & \leq    \|w(0)\|_{H}^2  
                + C \int_0^t   \|u\|^2_{V^5} \| w\|^2_{H} \dd s \\
             &  \quad + C \|F(p,U, y,A)\|^2_{L^\infty(J, V^{-\sigma})}
\end{align*}
for $t\in J$.
Therefore, using \eqref{E:6.8} and Grönwall's lemma, we obtain~\eqref{E:7.18}.

	 \addcontentsline{toc}{section}{Bibliography}
\newcommand{\etalchar}[1]{$^{#1}$}
\def\cprime{$'$} \def\cprime{$'$}
  \def\polhk#1{\setbox0=\hbox{#1}{\ooalign{\hidewidth
  \lower1.5ex\hbox{`}\hidewidth\crcr\unhbox0}}}
  \def\polhk#1{\setbox0=\hbox{#1}{\ooalign{\hidewidth
  \lower1.5ex\hbox{`}\hidewidth\crcr\unhbox0}}}
  \def\polhk#1{\setbox0=\hbox{#1}{\ooalign{\hidewidth
  \lower1.5ex\hbox{`}\hidewidth\crcr\unhbox0}}} \def\cprime{$'$}
  \def\polhk#1{\setbox0=\hbox{#1}{\ooalign{\hidewidth
  \lower1.5ex\hbox{`}\hidewidth\crcr\unhbox0}}} \def\cprime{$'$}
  \def\cprime{$'$} \def\cprime{$'$} \def\cprime{$'$}
\providecommand{\bysame}{\leavevmode\hbox to3em{\hrulefill}\thinspace}
\providecommand{\MR}{\relax\ifhmode\unskip\space\fi MR }
\providecommand{\MRhref}[2]{%
  \href{http://www.ams.org/mathscinet-getitem?mr=#1}{#2}
}
\providecommand{\href}[2]{#2}

\end{document}